\documentclass[12pt]{article}

\usepackage[margin=1.15in]{geometry}
\usepackage{sectsty}
\usepackage{xcolor}
\usepackage{amssymb,amsmath}
\usepackage{amsthm}
\usepackage{mathrsfs}
\usepackage{textcomp}
\usepackage{hyperref}

\hypersetup{
colorlinks,%
citecolor=black,%
filecolor=black,%
linkcolor=black,%
urlcolor=black
}

\usepackage[abbrev,nobysame,non-compressed-cites]{amsrefs}
\usepackage{todonotes}

\usepackage{verbatim}
\numberwithin{equation}{section}

\makeatletter

\newdimen\bibspace
\makeatother

\makeatletter

\newtheorem{thm}{Theorem}[section]
\newtheorem{lem}[thm]{Lemma}
\newtheorem{prop}[thm]{Proposition}

\newtheorem{cor}[thm]{Corollary}
\newtheorem{rem}[thm]{Remark}

\def\XXint#1#2#3{{\setbox0=\hbox{$#1{#2#3}{\int}$}
\vcenter{\hbox{$#2#3$}}\kern-.5\wd0}}

\def\C{\mathscr{C}}

\newcommand{\al}{\alpha}

\newcommand{\om}{\Omega}
\newcommand{\pa}{\partial}
\newcommand{\va}{\varepsilon}

\newcommand{\be}{\begin{equation}}
\newcommand{\ee}{\end{equation}}

\newcommand{\ol}{\overline}

\newcommand{\R}{\mathbb{R}}

\newcommand{\wt}{\widetilde}

\allowdisplaybreaks

\begin{document}

\title{\textbf{Schauder estimates for the linearized very fast diffusion equations in bounded domains}
\bigskip}

\author{Tianling Jin\footnote{T. Jin was partially supported by NSFC grant 12122120, and Hong Kong RGC grants GRF 16304125, GRF 16303624 and GRF 16303822.}, \quad Xushan Tu, \quad
Jingang Xiong\footnote{J. Xiong was partially supported by NSFC grants 12325104.}, \quad Zhen Zheng}

\date{\today}

\maketitle
\begin{abstract}
We establish Schauder estimates for linearized very fast diffusion equations with Dirichlet boundary conditions in bounded smooth domains. These estimates provide an important ingredient for deriving optimal boundary regularity and long-time dynamics of solutions to very fast diffusion equations.

\medskip
\noindent{\it Keywords}:   Schauder estimates, degenerate equation.

\medskip

\noindent {\it MSC (2010)}: Primary 35B65; Secondary 35K20, 35K65.

\end{abstract}


\section{Introduction}\label{sec:introduction}
Let \(\Omega\subset\mathbb R^n\), \(n\ge1\), be a bounded domain with smooth boundary, and let \(T>0\). We study global Schauder estimates for the degenerate Cauchy--Dirichlet problem
\begin{equation}\label{eq:general linearized equation}
\begin{cases}
Lv:=v_t-\omega^\alpha\bigl(a^{ij}v_{ij}+b^iv_i\bigr)+cv
=\omega^\beta f,
& \text{in }\Omega\times (0,T],\\
v=0,
& \text{on }\partial\Omega\times (0,T],\\
v(\cdot,0)=g_0,
& \text{in }\overline\Omega,
\end{cases}
\end{equation}
where \(\alpha,\beta\in\mathbb R\), $i,j=1,\cdots,n$, the coefficients \(a^{ij}\), \(b^i\), \(c\), and \(f\) depend on \((x,t)\), and the summation convention is used. The weight \(\omega\in C^3(\overline\Omega)\) is assumed to be comparable to the boundary distance
\[
d(x):=\operatorname{dist}(x,\partial\Omega),
\]
in the sense that, for some \(\lambda\in(0,1]\),
\begin{equation}\label{eq:assumptions of omega}
\lambda d \leq \omega \leq \lambda^{-1} d,
\qquad
\|\omega\|_{C^3(\overline{\Omega})}\leq 1.
\end{equation}
We assume that \(\{a^{ij}\}\) is symmetric and uniformly elliptic, and that
\begin{equation}\label{eq:structure_conditions}
|a^{ij}|+|b^i|+|c|+|f|\leq \lambda^{-1},
\qquad
a^{ij}\xi_i\xi_j\geq \lambda|\xi|^2
\quad \text{for all } \xi\in\mathbb R^n
\end{equation}
in \(\Omega\times(0,T]\).

The motivation for this study comes from the very fast diffusion equation with infinite boundary values:
\begin{equation}\label{eq:infty boundary}
\begin{cases}
u_t-\Delta\left(\dfrac{u^m}{m}\right)=0,
& \text{in }\Omega\times (0,+\infty),\\
u=+\infty,
& \text{on }\partial\Omega\times (0,+\infty),\\
u(\cdot,0)=u_0\ge0,
& \text{in }\Omega,
\end{cases}
\end{equation}
where \(m<0\).  Daskalopoulos--del Pino~\cite{Daskalopoulos1997nonlinear} proved existence of nonnegative solutions with infinite boundary values for arbitrary nonnegative initial data in $L^\infty(\Omega)$. By contrast, the very fast diffusion equation with zero boundary data admits no continuous weak solutions, and in the whole space there are no nontrivial nonnegative solutions in
$C([0,T];L^1_{\mathrm{loc}}(\mathbb R^n))$; see V\'azquez~\cite{vazquez1992nonexistence}.

After the change of variables $v=u^m$ and $p=1/m$, problem \eqref{eq:infty boundary} becomes
\begin{equation}\label{eq:equation for v}
\begin{cases}
(v^p)_t-p\Delta v=0,
& \text{in }\Omega\times [0,+\infty),\\
v=0,
& \text{on }\partial\Omega\times (0,+\infty),\\
v(\cdot,0)=u_0^m,
& \text{in }\Omega.
\end{cases}
\end{equation}
When $p>1$, this is the fast diffusion equation, whose solutions become extinct in finite time. The extinction behavior of solutions has been studied by, e.g., 
Berryman--Holland~\cite{berryman1980stability}, 
Feireisl--Simondon~\cite{FS}, 
Bonforte--Grillo--V\'azquez~\cite{BGV},
Bonforte--Figalli~\cites{BF21,BF24},
Akagi~\cite{Akagi}, 
Jin--Xiong~\cites{JX1,JX2, JX2023, JX26}, 
Choi--McCann--Seis~\cite{CMS}, and
Choi--Seis~\cite{ChoiS}. DiBenedetto--Kwong--Vespri~\cite{dibenedetto1991local} proved that the solution $v$ satisfies the global Harnack inequality
\begin{equation}\label{eq:global harnack inequality}
0<\inf_{\Omega}\frac{v}{d}
\leq \sup_{\Omega}\frac{v}{d}<\infty
\end{equation}
up to the extinction time.  Building on \eqref{eq:global harnack inequality}, Jin--Xiong~\cites{JX1,JX2} established the optimal regularity of solutions to \eqref{eq:equation for v}, as well as their extinction behavior in regular norms. When $0<p<1$, equation \eqref{eq:equation for v} becomes the porous medium equation. It is a slow diffusion equation, which means that if the initial data is nonzero and compactly supported in $\Omega$, then the solution remains compactly supported in $\Omega$ at least for a short time. Nevertheless, Aronson--Peletier~\cite{aronson1981large} showed that the solutions satisfy the global Harnack inequality \eqref{eq:global harnack inequality} after a waiting time, and they also obtained the asymptotic behavior of solutions. Jin--Ros-Oton--Xiong~\cite{JRX} established  the optimal regularity of  solutions after this waiting time and obtained fine asymptotics of the solutions.

The linearization of \eqref{eq:equation for v} at solutions satisfying \eqref{eq:global harnack inequality} naturally leads to equations of form \eqref{eq:general linearized equation}, with \(\alpha=1-p\). In this paper, we focus on the range
\[
\alpha\in(1,2),
\qquad
\beta\in(\alpha-1,\alpha/2],
\]
which corresponds to the very fast diffusion regime \(p\in(-1,0)\) in \eqref{eq:equation for v}. The optimal regularity and long-time dynamics of solutions to \eqref{eq:equation for v} in this range will be studied in our forthcoming work~\cite{JTXZ}. Here, we isolate the important analytic component: boundary Schauder estimates for the linearized degenerate problem \eqref{eq:general linearized equation}, which will be used later in our work \cite{JTXZ}.

Throughout the paper, classical solutions are understood as functions
\[
v\in C^{2,1}_{x,t}(\Omega\times(0,T])
\cap C(\overline\Omega\times[0,T])
\]
satisfying \eqref{eq:general linearized equation} pointwise in the interior and continuously up to the parabolic boundary.  We assume, in addition, that \(a^{ij}\), \(b^i\), \(c\), and \(f\) are locally H\"older continuous in \(\Omega\times(0,T]\), and that
\[
g_0\in C_0(\overline{\Omega})
:=\{g\in C(\overline{\Omega}) : g=0 \text{ on } \partial\Omega\}.
\]
The Schauder estimates we obtain below for \eqref{eq:general linearized equation} involve the weighted H\"older spaces \(\C^{\gamma}\) and \(\C_{\alpha,\beta}^{2+\gamma}\), whose precise definitions will be given in Section \ref{sec:notation and function spaces}.

\begin{thm}\label{thm: existence and uniqueness of general linear equation}
Let $\Omega \subset \mathbb{R}^n$ be a bounded domain with smooth boundary. Let
\[
\alpha \in (1,2), \quad \beta \in (\alpha-1, \alpha/2]\quad \text{and} \quad  \gamma^* := \frac{2-2\alpha+2\beta}{2-\alpha}.
\]
Let $T>0$.  Suppose for some  exponent $\gamma \in (0, \gamma^*)$, the coefficients satisfy \eqref{eq:assumptions of omega}, \eqref{eq:structure_conditions} and 
\[
\|a^{ij}\|_{\C^\gamma(\Omega\times [0,T])} + \|b^i\|_{\C^\gamma(\Omega\times [0,T])} + \|c\|_{\C^\gamma(\Omega\times [0,T])} \le \lambda^{-1}.
\]
Assume also that $f \in \C^\gamma(\Omega \times [0,T])$. Then, for any $g_0\in C_0(\ol{\om})$, there exists a unique classical solution $v\in \C_{\alpha,\beta}^{2+\gamma}(\Omega\times (0,T])\cap C(\overline\Omega\times[0,T])$ to the initial-boundary value problem \eqref{eq:general linearized equation}.
Moreover, the following estimates hold:
\begin{enumerate}
\item[(i)] For any $t_0 \in (0,T]$, there exists $C_{T,t_0} > 0$, which depends only on $n$, $\alpha$, $\beta$, $\gamma$, $\lambda$, $t_0$, $T$, and $\|\partial\Omega\|_{C^{2,(1-\alpha/2)\gamma}}$, such that
\[
\|v\|_{\C_{\alpha,\beta}^{2+\gamma}(\Omega\times (t_0,T])}
\le C_{T,t_0}\left( \|g_0\|_{L^{\infty}(\om)}+\|f\|_{\C^\gamma(\Omega\times [0,T])}\right).
\] 
Moreover,  the quantity $\omega^{-\beta}v_t$ and $\omega^{\alpha-\beta}D^2 v$ are H\"older continuous up to the boundary, and their boundary traces satisfy
\begin{equation}\label{eq:compatibility condition}
\omega^{-\beta}v_t=0,\quad   \omega^{\alpha-\beta}D^2 v =-\frac{f}{a^{\nu \nu }} \nu \otimes \nu \quad \text{on } \partial\Omega\times(0,T],
\end{equation}
where $\nu$ denotes the unit inner normal vector field on $\partial\Omega$  and
$a^{\nu\nu}:=a^{ij}\nu_i\nu_j$.
\item [(ii)] 
If, in addition,  $g_0 \in \C^{2+\gamma}_{\alpha,\beta}(\Omega)$ and satisfies the compatibility condition
\begin{equation}\label{eq:compatibility condition t=0}
\omega^{\alpha-\beta} a^{ij}D_{ij}g_{0}=-f  \quad \text{on }\pa\om\times \{t=0\},
\end{equation}
then we have $v\in \C^{2+\gamma}_{\al,\beta}(\overline{\Omega}\times [0,T])$, and the following estimate holds:
\[
\|v\|_{\C_{\alpha,\beta}^{2+\gamma}(\Omega\times [0,T])}
\le C_T\left( \|g_0\|_{\C^{2+\gamma}_{\al,\beta}(\om)}+\|f\|_{\C^\gamma(\Omega\times [0,T])}\right),
\] 
where $C_{T} > 0$, depending only on $n$, $\alpha$, $\beta$, $\gamma$, $\lambda$, $\|\partial\Omega\|_{C^{2,(1-\alpha/2)\gamma}}$, and $T$.
\end{enumerate} 
\end{thm}

\begin{rem}
As in Lee and Yun \cite{lee2025boundary}, one readily checks that
\[
v(x,t)=\frac{x_n^{2-\alpha}}{(2-\alpha)(\alpha-1)}, 
\]
which is only H\"older continuous up to the  boundary $\{x_n=0\}$ if $\alpha\in(1,2)$, solves $v_t-x_n^{\alpha}\Delta v=1$. Indeed, for a generic non-vanishing source term $f$, the condition $\beta > \alpha - 1$ is essential for obtaining $C^1$ regularity up to the boundary in Theorem \ref{thm: existence and uniqueness of general linear equation}. 
To see this heuristically, let us consider a simplified one-dimensional version of the operator near the boundary $x=0$, namely 
\[
v_t - x^\alpha v_{xx} = x^\beta f \quad \text{in } x >0, \quad  v=0 \quad \text{at}~x=0.
\]
As established in Section~\ref{sec: boundary holder regularity}, the solution $v$ satisfying $v = O(x^{2-\alpha+\beta})$ near $x=0$. Assuming that the solution $v$ is globally $C^1$, differentiating the equation with respect to $t$ shows that the time derivative $v_t$ satisfies a similar governing equation. Consequently,  $v_t = O(x^{2-\alpha+\beta})$.  
Dividing the equation by $x^{\beta}$, we then obtain
\[
v_{xx} =x^{\beta-\al} (-f+ x^{-\beta} v_t) =x^{\beta-\al} ( -f+O (x^{2-\alpha})) .
\]
Since the solution is $C^1$, this indicates the local integrability of $x^{\beta-\alpha} f$ near $x=0$. For a continuous source term $f$ with $f(0,1) \neq 0$, this condition necessitates that $\beta > \alpha - 1$.
\end{rem}

\begin{rem}
Note that \eqref{eq:compatibility condition} directly yields $\omega^{\alpha-\beta}a^{ij}D_{ij}v = -f$ on $\partial\Omega$. Consequently, the corner compatibility condition \eqref{eq:compatibility condition t=0} at the bottom boundary $\partial\Omega \times \{0\}$ is necessary for obtaining Schauder type estimates up to the initial time. 
\end{rem}

Since $L$ is uniformly parabolic away from \(\partial\Omega\), the main
difficulty is the boundary degeneracy caused by the factor \(\omega^\alpha\).
This difficulty is particularly pronounced in the range \(\alpha\in(1,2)\),
where even simple model equations in the above show that a non-vanishing source term may
prevent \(C^{1}\)-regularity up to the boundary. The additional vanishing
factor \(\omega^\beta\) in the right-hand side, with \(\beta>\alpha-1\), is
therefore essential: it gives the correct boundary growth of solutions, allows
one to recover Lipschitz bounds, and makes the weighted second-derivative
estimates finite. 
To capture this structure, we introduce H\"older spaces
adapted both to the intrinsic degeneracy of the operator, governed by
\(\omega^\alpha\), and to the boundary vanishing of the source term, governed
by \(\omega^\beta\). The natural second-order quantities in our setting are not \(v_t\) and \(D^2v\), but rather
$\omega^{-\beta}v_t$ and $\omega^{\alpha-\beta}D^2v$.
Accordingly, our Schauder norm contains the weighted terms
\[
    \bigl\|\omega^{-\beta}v_t\bigr\|_{\C^\gamma}
    +
    \bigl\|\omega^{\alpha-\beta}D^2v\bigr\|_{\C^\gamma},
\]
which reflects the interaction between the degeneracy \(\omega^\alpha\) and
the source vanishing \(\omega^\beta\). The proof then combines boundary growth
estimates, intrinsic rescaling, perturbative estimates, and localization near
\(\partial\Omega\).

There is by now a substantial literature on degenerate elliptic and parabolic equations of  type similar to \eqref{eq:general linearized equation}. In nondivergence form, Dong--Phan--Tran~\cite{dong2024nondivergence} studied second-order degenerate parabolic equations in the half-space under partially weighted VMO assumptions, and analogous divergence-form results were obtained in~\cite{dong2023degenerate}. Further results for singular or degenerate divergence-form parabolic equations and systems include
Bekmaganbetov--Dong~\cite{bekmaganbetov2025singulardegenerate},
Dong--Jeon~\cites{dong2025schaudertypeestimatesdegenerate,dong2026degenerate},
Dong--Phan~\cites{dong2021parabolic, dong2023parabolic},
and Jin--Xiong~\cite{JX3}.
For the critical case $\alpha=2$, Dong--Ryu~\cite{dong2025nondivergence} proved unique solvability in weighted mixed-norm Sobolev spaces, with related divergence-form results in Dong--Ryu~\cite{dong2026sobolev}. For $\alpha=1$, Daskalopoulos--Hamilton~\cite{daskalopoulos1998regularity} established existence, uniqueness and regularity of the solutions, and free-boundary regularity for the porous medium equation in two dimensions; related results for $\alpha\in(0,1)$ were obtained by Kim--Lee~\cite{kim2009smooth}. Kim--Lee--Yun \cites{kim2023higher,kim2024generalized} developed higher-order generalized Schauder theories for some related degenerate equations in which the degeneracy occurs only in the mixed and normal derivatives involving the normal variable, covering different ranges of the degeneracy exponent, for $1<\alpha< 2$ and $\alpha\leq 1$, respectively.

This paper is organized as follows: in Section \ref{sec:notation and function spaces}, we introduce some notations and the weighted H\"older spaces designed for our equation \eqref{eq:general linearized equation}. In Section \ref{sec: boundary holder regularity}, we consider the regularized approximating equation and obtain corresponding estimates. In Section \ref{sec:schauder regularity on halc cylinder}, we prove boundary Schauder estimates for model equations in a half-cylinder, first for constant coefficients and then for variable coefficients by perturbation.  In Section \ref{sec:proofofmainthm}, we pass to general bounded domains and prove Theorem \ref{thm: existence and uniqueness of general linear equation}.

\medskip

\textbf{AI statement:} All the mathematical results, overall proof architecture, and key arguments were developed by the authors. During later-stage revisions, AI tools were used to assist with editing and with simplifying some auxiliary arguments. All AI-assisted modifications were independently reviewed and verified by the authors, who take full responsibility for the final manuscript.

\section{Notations and function spaces}\label{sec:notation and function spaces}

In Sections \ref{sec: boundary holder regularity} and \ref{sec:schauder regularity on halc cylinder}, we restrict our attention to the special case where the spatial domain is a subset of the upper half-space $\R^n_+ := \left\{x = (x', x_n) \in \R^n \mid\ x_n > 0\right\}$. In this setting, we study the classical solution $u$ of the equation
\[
L u = f,
\]
where the degenerate parabolic operator $L$ is defined as
\begin{equation}\label{eq:typical operator Q1}
Lv :=  v_t - x_n^\alpha (a^{ij}v_{x_i x_j} + b^i v_{x_i}) + cv
\end{equation}
(which will be generalized to $Lv := v_t  - \omega(x)^\alpha (a^{ij}v_{x_i x_j} + b^i v_{x_i}) + cv$ later in Section \ref{sec:proofofmainthm}).

Unless otherwise specified, by a solution $u$ we mean a classical solution. That is, for a given space-time domain $Q$, the solution satisfies $u \in C^{2,1}_{x,t}(Q) \cap C(\overline{Q})$; it possesses continuous second-order spatial derivatives and a continuous first-order time derivative in the interior, and remains continuous up to the boundary.

For $x_0\in \R^n$ and $r>0$, we define the standard Euclidean open ball and half-ball as
\[
B_r(x_0):= \left\{x\in\R^n\mid |x-x_0|<r\right\}, \quad 
B_r^+(x_0):=B_r(x_0)\cap\left\{x_n > 0\right\}.
\]
For $X_0=(x_0,t_0)\in \R^n\times\R$, we define the standard parabolic cylinder and half-cylinder adapted to the degenerate boundary scaling as
\[
Q_r(X_0):=B_r(x_0)\times (t_0-r^{2-\al},t_0], \quad Q_r^+(X_0):=Q_r(X_0)\cap\{x_n > 0\}.
\]
We denote by $\pa_pQ_r(X_0)$ the parabolic boundary of $Q_r(X_0)$, defined by
\[
\pa_p Q_r(X_0):=\left(\pa B_r(x_0)\times [t_0-r^{2-\al}, t_0 ]\right)
\bigcup
\left(B_r(x_0)\times \left\{t=t_0-r^{2-\al}\right\}\right).
\]
For simplicity, we omit $x_0$ or $X_0$ when they are the origin or $(0,1)$.

\subsection{Intrinsic parabolic metric and associated function spaces}\label{sec:metric and holder space}

Following the approach of Daskalopoulos--Hamilton
\cite{daskalopoulos1998regularity} for the case $\alpha=1$ and that of
Kim--Lee \cite{kim2009smooth} for the case $\alpha\in(0,1)$, we equip
$\mathbb R^n_+$ with an intrinsic Riemannian metric adapted to the degeneracy
of the operator \eqref{eq:typical operator Q1} near the boundary $\{x_n=0\}$.
More specifically, for the model degenerate parabolic operator
\[
L_0 v := v_t - x_n^\alpha \Delta v,
\]
the metric is defined by the line element
\[
d\ell^2=\frac{dx_1^2+dx_2^2+\cdots+dx_n^2}{2x_n^\alpha}.
\]
Let $s(x,y)$ denote the geodesic distance on $\mathbb{R}^n_+$ induced by this metric.
Then, $s(x,y)$ is comparable to 
\[
\ol{s}[x,y]
=\frac{ |x-y|}
{|x_n|^{\frac{\al}{2}}+|y_n|^{\frac{\al}{2}}
+|x -y|^{\frac{\al}{2}}}
\]
in the sense 
\[
c(n,\alpha)\bar s(x,y)\le s(x,y)\le C(n,\alpha)\bar s(x,y).
\]
The following parabolic distance on $\mathbb{R}^n_+\times\mathbb{R}$ reflects the natural space--time scaling of the operator:
\[
s\bigl[(x ,t_1),(y,t_2)\bigr]
:=\max \left\{s[x,y], \sqrt{|t_1-t_2|}\right\}.
\]

\begin{rem}\label{rem:metric_comparison}
The relationship between the standard cylinder $Q_r(X_0)$ and the intrinsic distance $s$ can be intuitively understood via the following local equivalences:
\begin{itemize}
    \item \textbf{At a boundary point} ($X_0 \in \{x_n = 0\}$): 
    \[
    X \in \pa_p Q^+_r(X_0) \iff s[X, X_0] \approx r^{1-\frac{\alpha}{2}}.
    \]
    \item \textbf{At a strictly interior point} ($X_0$ with $x_{0,n} = d > 0$, for $r \leq 7d/8$): 
    \[
    X \in Q_r(X_0) \iff s[X, X_0] \approx  x_{0,n}^{-\frac{\alpha}{2}}  |x - x_0| + \sqrt{|t - t_0|}.
    \]
\end{itemize}
\end{rem}

Endowed with this parabolic metric $s$, we denote by $\C^\gamma(Q) := \C_{\mathbb{R}^n_+}^\gamma(Q)$ the space of H\"older continuous functions on a space-time domain $Q \subset \mathbb{R}^n_+ \times [0,+\infty)$. For any real-valued function $u$ defined on $Q$ and an exponent $\gamma \in (0,1)$, the corresponding H\"older seminorm and full norm are defined, respectively, by
\[
[u]_{\C^\gamma(Q)}
=\sup_{\substack{P_1,P_2\in Q\\ P_1\neq P_2}}
\frac{|u(P_1)-u(P_2)|}{s[P_1,P_2]^\gamma},
\]
and
\[
\|u\|_{\C^{\gamma}(Q)}
=\|u\|_{L^\infty(Q)}+[u]_{\C^\gamma(Q)}.
\]
With this norm, the space $\C^\gamma(Q)$ of H\"older continuous functions on $Q$ (with respect to the metric $s$) is a Banach space.

For $\al\in (1,2)$ and $\beta\in (\al-1,\al/2]$, we then define  $\C_{\al,\beta}^{2+\gamma}(Q):=\C_{\al,\beta,\R_+^n}^{2+\gamma}(Q)$
to be the Banach space of all functions $u$ on $Q$ such that the norm
\[
\|u\|_{\C_{\al,\beta}^{2+\gamma}(Q)}
=\|u\|_{\C^\gamma(Q)}
+\sum_{i=1}^{n}\|u_{x_i}\|_{\C^\gamma(Q)}
+\|x_n^{-\beta}u_t \|_{\C^\gamma(Q)}
+\sum_{1\leq i\leq j\leq n}\|x_n^{\al-\beta} u_{x_ix_j}\|_{\C^\gamma(Q)}
\]
is finite. 

For comparison, we recall the classical H\"older spaces in the space and time variables. For $0<\gamma<1$, we write
\[
C^\gamma(Q):=C_x^\gamma(Q)\cap C_t^{\gamma/2}(Q),
\]
and
\[
C^{2,\gamma}(Q):=C_x^{2,\gamma}(Q)\cap C_t^{1,\gamma/2}(Q).
\]

In Section \ref{sec:proofofmainthm},  we study the problem on a general bounded domain $\Omega \subset \mathbb{R}^n$ with a smooth boundary. To properly formulate Theorem \ref{thm: existence and uniqueness of general linear equation} and accurately capture the degenerate behavior of the operator 
$
Lv := v_t - \omega^\alpha \left( a^{ij} v_{ij} + b^i v_i \right) + c v
$
near $\partial \Omega$, we now extend the aforementioned definitions of the metric $s$ and associated function spaces $\C$ from $\R_+^n$ to the whole domain $\Omega$ as follows.

Away from the boundary, the distance function $s_{\Omega}$ on $\ol{\Omega}$ is equivalent to the standard Euclidean distance. Near the boundary, $s_{\Omega}$ is defined by pulling back the corresponding distance $s$ on $\mathbb{R}^n_+$, through local boundary-flattening maps $\varphi:\mathbb{R}^n_+\to \Omega$.
The distance $s_{\Omega}$ is comparable to the quantity
\[
\ol{s}_{\Omega}[x,y]
=\frac{|x-y|}
{d(x)^{\frac{\al}{2}}+d(y)^{\frac{\al}{2}}
+|x-y|^{\frac{\al}{2}}}.
\]
The parabolic distance in the cycloidal metric is given by
\[
\ol{s}_{\Omega}[(x ,t_1),(y,t_2)]=\max \left\{s_{\Omega}[x,y],~\sqrt{|t_1-t_2|}\right\}.
\] 
By abuse of notation, we write $s$ and $\ol{s}$ for $s_{\Omega}$ and $\ol{s}_{\Omega}$.

Let $\widetilde{Q}$ be a subset of the cylinder $\om\times [0,T]$. We denote by $\C^\gamma(\widetilde{Q}):=\C_{\Omega}^\gamma(\widetilde{Q})$ the space of
H\"{o}lder continuous functions on $\widetilde{Q}$ with respect to the metric $s_{\Omega}$. We define $\C_{\alpha,\beta}^{2+\gamma}(\widetilde{Q}):=\C_{\alpha,\beta,\Omega}^{2+\gamma}(\widetilde{Q})$ to be the space of all functions
$w$ on $\widetilde{Q}$ such that $w$, $w_{x_i}$, $d^{-\beta}w_t$  and $d^{\alpha-\beta}w_{x_ix_j}$ admit continuous extensions up to the boundary of $\widetilde{Q}$, and these extensions
belong to $\C^\gamma(\widetilde{Q})$. Here $i,j\in\{1,\dots,n\}$ and $d$ denotes the distance to $\pa\om$.
Both $\C^\gamma(\widetilde{Q})$ and $\C_{\alpha,\beta}^{2+\gamma}(\widetilde{Q})$ are Banach spaces under the norms $\|\cdot\|_{\C^\gamma(\widetilde{Q})}$ and
\[ 
\|w\|_{\C_{\al,\beta}^{2+\gamma}(\widetilde{Q})}=\|w\|_{\C^\gamma(\widetilde{Q})}+\sum_{i=1}^{n}\|w_{x_i}\|_{\C^\gamma(\widetilde{Q})}+\|d^{-\beta}w_t\|_{\C^\gamma(\widetilde{Q})}
+\sum_{1\le i\le j\le n}\|d^{\al-\beta}w_{x_ix_j}\|_{\C^\gamma(\widetilde{Q})}. 
\]
For simplicity, in the statements below we omit $\al$ and $\beta$ from the subscript notation. 

A function $g\in \C_{\al,\beta}^{\gamma}(\Omega)$ or $\C_{\al,\beta}^{2+\gamma}(\Omega)$ is regarded, respectively, as an element of $\C_{\al,\beta}^{2+\gamma}(\widetilde{Q})$ or $\C_{\al,\beta}^{2+\gamma}(\widetilde{Q})$ by identifying it with its time-independent extension to $\widetilde{Q}$.
\bigskip

\subsection{Equivalent forms of the Hölder norm}
\label{sec:holder}
In this section, we focus on the equivalent H\"{o}lder norm on $\R^n_+$. For a function $v \in C(\overline{Q}_1^+)$, we introduce the boundary pointwise H\"{o}lder seminorm
\[
[v]_{\C_{bd}^{\gamma}(Q_1^+)} := \sup_{\substack{z_0 \in \partial_p Q_1^+ \cap \{x_n = 0\} \\ z_1 \in Q_1^+,\ z_1 \neq z_0}} \frac{|v(z_1) - v(z_0)|}{s[z,z_0]^\gamma},
\]
and the weighted interior Euclidean seminorm
\[
[v]_{\C_{int}^{\gamma}(Q_1^+)} := \sup_{\substack{z_0, z_1 \in Q_1^+ \\ 0 < s[z_1, z_0] \le \frac{1}{2}x_{0,n}^{1-\frac{\alpha}{2}}}} \frac{|v(z_1) - v(z_0)|}{x_{0,n}^{-\frac{\alpha}{2}\gamma} |x_1 - x_0|^\gamma + |t_1 - t_0|^{\frac{\gamma}{2}}},
\] 
where $z_0=(x_0,t_0)$, $z_1=(x_1,t_1)$, and $x_{0,n}$ denotes the $n$-th coordinate of $z_0$. 

\begin{lem} 
\label{lem:euclidean_patching}
Suppose $v \in C(\overline{Q}_1^+)$ such that $[v]_{\C_{bd}^{\gamma}(Q_1^+)} +[v]_{\C_{int}^{\gamma}(Q_1^+)} < \infty$. Then, $v \in \C^\gamma(\overline{Q}_1^+)$ with respect to the intrinsic parabolic distance $s[\cdot, \cdot]$, and we have the uniform global estimate
\[
[v]_{\C^\gamma(Q_{1}^+)} \le C(n, \alpha, \gamma) \left( [v]_{\C_{bd}^{\gamma}(Q_{1}^+)} + [v]_{\C_{int}^{\gamma}(Q_{1}^+)} \right).
\]
\end{lem}

\begin{proof}
Let $z_0, z_1 \in \overline{Q}_1^+$ and denote $r =  x_{0,n}$.
We divide the proof into two cases.

\textbf{Case 1:} $s[z_0, z_1] \ge \frac{1}{2}r^{1-\frac{\alpha}{2}}$. 
Let us write $z_0 = \bar{z}_0 + x_{0,n}e_n$ for some boundary projection point $\bar{z}_0 \in \partial_p Q_1^+ \cap \{x_n = 0\}$. The triangle inequality and the definition of $[v]_{\C_{bd}^{\gamma}}$ then yield:
\[
|v(z_0) - v(z_1)| \le |v(z_0) - v(\bar{z}_0)| + |v(z_1) - v(\bar{z}_0)| \le (2^{\gamma+1}+4^{\gamma+1})[v]_{\C_{bd}^{\gamma}} s[z_0, z_1]^\gamma.
\]  

\textbf{Case 2:} $s[z_0, z_1] \le \frac{1}{2}r^{1-\frac{\alpha}{2}}$. 
In this strictly interior regime, we have
\[
|v(z_0) - v(z_1)| \le  [v]_{\C_{int}^\gamma} s[z_0, z_1]^\gamma  \leq C  [v]_{\C_{int}^\gamma}  \left(r^{-\frac{\alpha}{2}\gamma} |x_0 - x_1|^\gamma + |t_0 - t_1|^{\frac{\gamma}{2}}\right).
\]
Combining both cases, we establish that
\[
[v]_{\C^\gamma(Q_{1}^+)} \le C(n, \alpha, \gamma) \left( [v]_{\C_{bd}^{\gamma}(Q_{1}^+)} + [v]_{\C_{int}^{\gamma}(Q_{1}^+)} \right).
\]
\end{proof}

\subsection{Parabolic scaling and metric homogeneity}
\label{sec:scaling}
The structure of the degenerate operator and the intrinsic metric $s[\cdot, \cdot]$ are fundamentally determined by the natural scaling of the equation near the boundary $\{x_n = 0\}$. For a fixed boundary point $z_0 = (x', 0, t_1)$ and a scale factor $r > 0$, we consider the parabolic rescaling
\be\label{eq:change of variables}
x = (x', 0) + ry, \quad t = t_1 + r^{2-\alpha} (\tau-1).
\ee
To preserve the non-homogeneous structure of the equation (i.e., keeping the right-hand side invariant), we scale the function $v$ by the potential factor $r^{2-\alpha}$:
\begin{equation}\label{eq:scale at x_n=0}
v_r(y,\tau):=v_{r,z_0}(y,\tau) = \frac{v((x',0)+ry, t_1 + r^{2-\alpha}(\tau-1))}{r^{2-\alpha}}
\end{equation}
A direct computation shows that if $v$ satisfies the equation 
\[
\partial_t v - x_n^\alpha (a^{ij}v_{x_i x_j} + b^i v_{x_i}) + cv = f,
\]
then the rescaled function $v_r$ perfectly preserves the principal operator:
\be\label{eq:scaled version equation}
\partial_\tau v_r - y_n^\alpha (a^{ij}_r \partial_{y_i y_j} v_r +  b^i_r \partial_{y_i} v_r) + c_r v_r = f_r,
\ee
where 
\be\label{eq:scaled coefficients}
\begin{aligned}
a_r^{ij}(y,\tau)&=a^{ij} ((x',0)+ry, t_1 + r^{2-\alpha}(\tau-1)), \\
b_r^i(y,\tau)&=rb^i((x',0)+ry, t_1 + r^{2-\alpha}(\tau-1)),\\
c_r(y,\tau)&=r^{2-\alpha}c((x',0)+ry, t_1 + r^{2-\alpha}(\tau-1)), \\
f_r(y,\tau)&= f((x',0)+ry, t_1 + r^{2-\alpha}(\tau-1)).
\end{aligned}
\ee
The leading principal part of the operator is exactly preserved under this transformation, while the lower-order terms vanish in the blow-up limit as $r \to 0$ (since $1 < \alpha < 2$, implying $2-\alpha > 0$).

Crucially, the intrinsic parabolic distance scales homogeneously under this coordinate transformation. 
For any two points $z, \bar{z}$ corresponding to the rescaled points $\zeta=(y,\tau)$ and $\bar{\zeta}=(\bar{y},\bar{\tau})$, we have
\[
s[z, \bar{z}]  = r^{1-\frac{\alpha}{2}} s[\zeta, \bar{\zeta}].
\]
Let $A_r$ be the rescaled domain corresponding to $A$ under the change of variables \eqref{eq:change of variables}. Then the intrinsic H\"{o}lder norm and seminorm scale as:
\[
\|v\|_{L^\infty(A)} = r^{2-\alpha}\|v_r\|_{L^\infty(A_r)}, \quad [v]_{\C^\gamma(A)} = r^{(2-\gamma)(1-\frac{\alpha}{2})} [v_r]_{\C^\gamma(A_r)}.
\]
For each integer $k,l \ge 0$, the norms of derivatives scale as:
\[ 
\| D_x^k\partial_t^l v\|_{L^\infty(A)} = r^{(2-2l)(1-\frac{\alpha}{2}) - k} \| D_y^k\partial_\tau^l v_r\|_{L^\infty(A_r)}, 
\]
\[
[ D_x^k\partial_t^l v]_{\C^\gamma(A)} = r^{(2-2l-\gamma)(1-\frac{\alpha}{2}) - k} [ D_y^k\partial_\tau^l v_r]_{\C^\gamma(A_r)}.
\]
Moreover, for each real number $\sigma \in \mathbb{R}$, the weighted norms scale precisely as:
\[
\| x_n^{\sigma}D_x^k\partial_t^l v\|_{L^\infty(A)} = r^{\sigma + (2-2l)(1-\frac{\alpha}{2}) - k} \| y_n^{\sigma}D_y^k\partial_\tau^l v_r\|_{L^\infty(A_r)},
\]
\[
[ x_n^{\sigma}D_x^k\partial_t^l v]_{\C^\gamma(A)} = r^{\sigma + (2-2l-\gamma)(1-\frac{\alpha}{2}) - k} [ y_n^{\sigma}D_y^k\partial_\tau^l v_r]_{\C^\gamma(A_r)}.
\]

In addition, the intrinsic interior H\"{o}lder seminorm $[\cdot]_{\C^\gamma_{int}}$ of the function and its derivatives are bounded by the standard classical H\"{o}lder seminorms of the rescaled functions
\begin{lem}
\label{lem:interior_seminorm_scaling}
We have for each $0 < \rho < 1$, 
\begin{equation}\label{eq:seminorm_scaling}
[v]_{\C_{int}^{\gamma}(Q_\rho^+)} =\sup_{\substack{z_0 \in \partial_p Q_\rho^+ \cap \{x_n = 0\} \\ 0 < r < \rho }} r^{(2-\gamma)(1-\frac{\alpha}{2})} [v_{r,z_0}]_{C^\gamma(Q_{1/2}(0,1,1))}
\end{equation}
and for each integer $k,l \ge 0$ and real number $\sigma \in \mathbb{R}$,
\begin{equation}\label{eq:seminorm_scaling weight}
[x_n^{\sigma}D_x^k\partial_t^l v]_{\C_{int}^{\gamma}(Q_\rho^+)} = \sup_{\substack{z_0 \in \partial_p Q_\rho^+ \cap \{x_n = 0\} \\ 0 < r < \rho }} r^{\sigma + (2-2l-\gamma)(1-\frac{\alpha}{2}) - k}[y_n^{\sigma}D_y^k\partial_\tau^l  v_{r,z_0}]_{C^\gamma(Q_{1/2}(0,1,1))},
\end{equation}
where the right-hand sides are restricted to the subset of $Q_\rho^+$ where $v$ is well-defined.
\end{lem}
\begin{proof}
This estimate follows directly by substituting the scaling $v = r^{2-\alpha} v_{r,z_0}$ into the intrinsic H\"{o}lder quotient and algebraically factoring out the weight $r^{(1-\frac{\alpha}{2})\gamma}$ from the denominator.
\end{proof}

\begin{thm}[Classical parabolic regularity, \cite{lieberman1996second}]
\label{thm:standard_interior_boundary_estimates}
Let $U = \Omega_0 \times (t_0, t_1)$ be a local parabolic cylinder, and denote $V = U \cap Q_{3/4}(0,1,1)\subset \{x_n\geq 1/4 \}$. Suppose that $u$ is a classical solution to the equation
\[
\begin{cases}
u_t  - x_n^\alpha (a^{ij}u_{x_i x_j} + b^i u_{x_i}) + cu = \tilde{f} &  \text{in } V,\\
u = g & \text{on } W:=(\partial_p U) \cap Q_{3/4}(0,1,1),
\end{cases}
\]
where the coefficients satisfy the uniform boundedness and ellipticity assumptions \eqref{eq:structure_conditions} and we set $g\equiv 0$ if $W$ is empty.
\begin{enumerate}
    \item [(i)] \textbf{(H\"{o}lder Estimate)} If $\partial \Omega_0 \cap B_{3/4}(0,1)$ is either empty or is Lipschitz with its Lipschitz constant bounded by $\lambda^{-1}$, there exists a universal exponent $\gamma_0 \in (0,1)$, depending only on $n, \lambda$, and $\alpha$, such that $u \in C^{\gamma_0}(\overline{V \cap Q_{1/2}(0,1,1)})$ with
\begin{equation}\label{eq:classical holder}
    \|u\|_{C^{\gamma_0}(\overline{V \cap Q_{1/2}(0,1,1)})} \leq C \left( \|u\|_{L^\infty(V)} + \|g\|_{C^{\gamma_0}(W)} + \|\tilde{f}\|_{L^\infty(V)} \right),
\end{equation}
where $C = C(n, \lambda, \alpha) > 0$. Suppose additionally that $\|a_{ij}\|_{C^{\wt\gamma}(V)} \leq \lambda^{-1}$ for some $\tilde{\gamma} > 0$. Then, the Hölder estimates hold with $\gamma_0$ replaced by any $\wt\gamma \in (0,1)$, where the constant $C$ now also depends on $\wt\gamma$.
\item[(ii)] \textbf{(Schauder Estimate)}
Furthermore, suppose in addition that for a given $\gamma \in (0,1)$, the coefficients are uniformly H\"{o}lder continuous satisfying
\[
\|a^{ij}\|_{C^\gamma(V)}+\|b^i\|_{C^\gamma(V)}+\|c\|_{C^\gamma(V)}\le \lambda^{-1}.
\]
If $\partial \Omega_0 \cap B_{3/4}(0,1)$ is empty, then there exists a constant $C = C(n, \lambda, \alpha, \gamma)>0$ such that
    \begin{equation}\label{eq:classical schauder}
        \|u\|_{C^{2,\gamma}(\overline{V \cap Q_{1/2}(0,1,1)})} \le C \left( \|u\|_{L^\infty(V)} + \|g\|_{C^{2,\gamma}(W)} + \|\tilde{f}\|_{C^\gamma(V)} \right).
    \end{equation}
\end{enumerate}
\end{thm}

\section{Boundary H\"{o}lder regularity}\label{sec: boundary holder regularity}

In this section, assuming that the structure conditions \eqref{eq:structure_conditions} hold and that $g \in C(\overline{Q}_1^+)$ is a continuous function satisfying
\begin{equation}\label{eq:structure_conditions g}
g = 0 \quad \text{on} \quad \partial_p Q_1^+ \cap \{x_n=0\},
\end{equation}
we study the initial-boundary value problem
\begin{equation}\label{eq:equation v0 bcg}
\begin{cases}
v_t - x_n^\alpha \bigl(a^{ij}v_{ij} + b^i v_i\bigr) + cv = x_n^\beta f & \text{in} \quad Q_1^+, \\
v = g & \text{on} \quad \partial_p Q_1^+,
\end{cases}
\end{equation}
where $\alpha\in(1,2)$ and $\beta\in[0,\infty)$ are fixed parameters.

For later discussions in subsequent sections, we introduce a regularization parameter $\varepsilon \in [0,1]$ and consider the classical solution to the strictly parabolic approximation:
\begin{equation}\label{eq:equation approximate bcg}
\begin{cases}
L_{\varepsilon} v_{\varepsilon} = (x_n+\varepsilon)^\beta f & \text{in} \quad Q_1^+, \\
v_\varepsilon = g & \text{on} \quad \partial_p Q_1^+,
\end{cases}
\end{equation}
where
\[
L_{\varepsilon} u := u_t - (x_n+\varepsilon)^\alpha \bigl(a^{ij}u_{ij}+b^iu_i\bigr) + cu.
\]
Here, we shall denote $v := v_0$ as the classical solution to the original limit problem \eqref{eq:equation v0 bcg}.

Our goal is to demonstrate how the global continuity or H\"{o}lder regularity of the solution $v$ in $\overline{Q}_1^+$ is dictated by the regularity of the boundary data $g$---specifically, whether it is merely continuous or possesses higher H\"{o}lder regularity---along with the influence of $\beta$.

Note that for any $\varepsilon \in [0,1]$, by comparing $v_\varepsilon$ with the spatially independent barriers
\[
W^\pm(t) = \pm e^{\Lambda t} \left( \|g\|_{L^\infty(\partial_p Q_1^+)} + \|f\|_{L^\infty(Q_1^+)} \right),
\]
where the constant $\Lambda > 0$ is chosen sufficiently large, depending on $\|c\|_{L^\infty}$, and applying the maximum principle, we obtain
\[
\| v_\varepsilon \|_{L^\infty(Q_1^+)} \leq C\left( \|g\|_{L^\infty(\partial_p Q_1^+)} + \|f\|_{L^\infty(Q_1^+)} \right).
\]
Here, $C > 0$ is a constant independent of $\varepsilon \in [0,1]$ and $\beta \in [0,\infty)$.

Thus, without loss of generality, we may always assume the normalization
\begin{equation}\label{eq:normalizationg}
\| v_\varepsilon \|_{L^\infty(Q_1^+)} + \|g\|_{L^\infty(\partial_p Q_1^+)} + \|f\|_{L^\infty(Q_1^+)} \leq 1.
\end{equation}
Otherwise, we can simply normalize the problem by dividing the equation by the constant factor $C\left(\|g\|_{L^\infty(\partial_p Q_1^+)} + \|f\|_{L^\infty(Q_1^+)} + 1\right)$.

Before stating the main estimates, we introduce two quantities to characterize the uniform modulus of continuity of the boundary data $g$ near the degenerate boundary $\{x_n = 0\}$. We define the interior modulus of continuity as
\begin{equation}\label{eq:def_omega}
\kappa(s) := \|g\|_{L^\infty(\partial_p Q_1^+ \cap \{0 < x_n < s\})},
\end{equation}
and the corresponding modulus of continuity on the initial time slice as
\begin{equation}\label{eq:def_omega_0}
\kappa_0(s) := \|g(x,0)\|_{L^\infty(\partial_p Q_1^+ \cap \{0 < x_n < s\} \cap \{t=0\})}.
\end{equation}

\begin{thm}\label{thm:Existence and uniqueness}
Suppose that $\alpha\in(1,2)$ and $\beta\in[0,\infty)$. Assume further that the structural conditions \eqref{eq:structure_conditions} hold in $Q_1^+$, that the condition on $g$ in \eqref{eq:structure_conditions g} holds, and that the normalization \eqref{eq:normalizationg} holds. For each $\varepsilon\in[0,1]$, let $v_\varepsilon$ be the solution of \eqref{eq:equation approximate bcg}. Then there exists a constant $C>0$, depending only on $n$, $\alpha$, $\beta$, and $\lambda$, such that:
\begin{enumerate}
    \item[i)] when $\beta < \alpha-1$, we have
    \begin{equation} \label{eq:c0 estimate for vva}
    |v_\varepsilon(x,t)| \leq C\inf_{\tau>0} \left( \kappa(\tau) + \left(\sup_{ \delta> \tau } \frac{\kappa(\delta) }{\delta^{2-\alpha+\beta}} + 1\right) x_n^{2-\alpha+\beta} \right)
    \end{equation}
    in  $Q_1^+$ and 
    \begin{equation}\label{eq:c0 estimate for vva x'}
    |v_\varepsilon(x,t)| \leq C\inf_{\tau>0} \left( \kappa_0(\tau) + \left(\sup_{ \delta> \tau } \frac{\kappa_0(\delta) }{\delta^{2-\alpha+\beta}} + 1\right) x_n^{2-\alpha+\beta} \right)
    \end{equation}
    in  $Q_1^+\cap \{|x'|\leq 3/4\}$;
    \item[ii)] when $\beta > \alpha-1$, we have
    \begin{equation}\label{eq:c1 estimate for vva}
    |v_\varepsilon(x,t)| \leq C\inf_{\tau>0} \left( \kappa(\tau) + \left(\sup_{ \delta> \tau } \frac{\kappa(\delta) }{\delta} + 1\right) x_n \right)
    \end{equation}
     in  $Q_1^+$ and 
    \begin{equation}\label{eq:c1 estimate for vva x'}
    |v_\varepsilon(x,t)|  \leq C\inf_{\tau>0} \left( \kappa_0(\tau) + \left(\sup_{ \delta> \tau } \frac{\kappa_0(\delta) }{\delta} + 1\right) x_n \right) 
    \end{equation}
    in  $Q_1^+\cap \{|x'|\leq 3/4\}$.
\end{enumerate}
Here, $\kappa(s)$ and $\kappa_0(s)$ are defined as in \eqref{eq:def_omega} and \eqref{eq:def_omega_0}, respectively. 

As a corollary, the solution $v_\varepsilon$ of \eqref{eq:equation approximate bcg} converges uniformly in $\overline{Q}_1^+$ to a unique classical solution $v:=v_0 \in C_{x,t}^{2,1}(Q_1^+) \cap C(\overline{Q}_1^+)$ of \eqref{eq:equation v0 bcg} as $\varepsilon \to 0^+$.
\end{thm}
\begin{proof}
Recalling \eqref{eq:normalizationg}, it suffices to establish both estimates in the boundary neighborhood $\{0 < x_n < r_0\}$ for some sufficiently small $r_0 \in (0,1)$.

\textbf{Step 1.} 
We first assume $\beta < \alpha-1$ and prove \eqref{eq:c0 estimate for vva} and \eqref{eq:c0 estimate for vva x'}. 

For each small $\tau>0$, we consider the upper barriers
\[
\Phi_{\tau}^+(x,t) = e^{\Lambda t} \Big[ M_{\tau} Aw(x) +\kappa(\tau)\Big], 
\]
where 
\[
\Lambda=\|c\|_{L^{\infty}(Q_1^+)},\quad w(x)=x_n^{2-\alpha+\beta}\quad \text{and} \quad  M_{\tau}= \sup_{ \delta> \tau }  \frac{\kappa(\delta) }{\delta^{2-\alpha+\beta}} +1.
\]
Then, we have
\[
L_{\varepsilon} \Phi_{\tau}^+ = 
e^{\Lambda t}  M_{\tau}  A\bigl( L_{\varepsilon} w -cw \bigr) +(\Lambda + c)\Phi_{\tau}^+\geq e^{\Lambda t}  M_{\tau}A \bigl( L_{\varepsilon} w - cw \bigr).
\]
Evaluating the principal part and the gradient drift terms on $w$, we compute
\[
L_{\varepsilon} w - cw= c_{\alpha,\beta}a^{nn}(x_n+\varepsilon)^{\alpha}x_n^{-\alpha+\beta}  - (x_n+\varepsilon)^\alpha b^n (2-\alpha+\beta)x_n^{1-\alpha+\beta},
\]
where $c_{\alpha,\beta}=-(2-\alpha+\beta)(1-\alpha+\beta)$.
Since $0<2-\alpha+\beta<1$, by choosing $A$ sufficiently large and $r_0$ sufficiently small, we obtain
\[
A \bigl(L_{\varepsilon} w - cw\bigr)\geq  (x_n+\varepsilon)^\beta f \quad \text{in } Q_1^+\cap \{0 < x_n < r_0\}.
\]
Consequently,
\[
L_{\varepsilon} \Phi_{\tau}^+ \geq e^{\Lambda t}M_{\tau} (x_n+\varepsilon)^\beta  f \geq (x_n+\varepsilon)^\beta  f = L_{\varepsilon} v_\varepsilon \quad \text{in }   Q_1^+\cap \{0 < x_n < r_0\}.
\]

Note that for sufficiently large $A$, we have $\Phi_{\tau}^+ \geq A r_0^{2-\alpha+\beta} \geq v_\varepsilon$ on $\{x_n = r_0\}$, and
\[
\Phi_{\tau}^+ \geq \kappa(\tau)  \geq \kappa(x_n) \geq g = v_\varepsilon \quad \text{on } \partial_p Q_1^+ \cap \{0 \leq x_n < r_0\}.
\]
The comparison principle then yields
\[
v_\varepsilon(x,t) \leq \Phi_{\tau}^+(x,t) \quad \text{in } Q_1^+ \cap \{0 < x_n < r_0\}.
\] 
By symmetry, $v_\varepsilon(x,t) \geq -\Phi_{\tau}^+(x,t)$ holds as well. In conclusion,
\[
|v_\varepsilon(x,t)| \leq e^{\Lambda t} \inf_{\tau>0} \left( \kappa(\tau) + M_{\tau}A  x_n^{2-\alpha+\beta}    \right) \quad \text{in } Q_1^+\cap \{0 < x_n < r_0\}.
\]
This establishes \eqref{eq:c0 estimate for vva}.

Moreover, noting that $L_{\varepsilon} |x' - y'|^2- c|x' - y'|^2 \leq C(x_n+\varepsilon)^{\alpha}$, the same argument also shows that for every $(y',0) \in \pa B_{1/2}^+$, we have
\[
|v_\varepsilon(x,t)| \leq e^{\Lambda t} \inf_{\tau>0} \left(\kappa(\tau) + M_{\tau,0}A  x_n^{2-\alpha+\beta}  +16|x' - y'|^2  \right),
\]
where $M_{\tau,0}=\sup_{ \delta> \tau } \frac{\kappa_0(\delta) }{\delta^{2-\alpha+\beta}} + 1$. This establishes \eqref{eq:c0 estimate for vva x'}.

\textbf{Step 2.} 
We assume $\beta > \alpha-1$ and prove \eqref{eq:c1 estimate for vva} and \eqref{eq:c1 estimate for vva x'}. 
Let $0<\widetilde{\beta}<\min\{\beta-\alpha+1,1\}$ be a small constant. 
The proof for \eqref{eq:c1 estimate for vva} follows an identical argument by replacing the function with
\[
\widetilde{\Phi}_{\tau}^+(x,t) = e^{\Lambda t} \Big[\widetilde{M}_{\tau} A\widetilde{w}(x) + \kappa(\tau)\Big], 
\]
where 
\[
\Lambda=\|c\|_{L^{\infty}(Q_1^+)},\quad \widetilde{w}=x_n
-x_n^{1+\widetilde{\beta}} \quad \text{and} \quad  \widetilde{M}_{\tau}= \sup_{ \delta> \tau }  \frac{\kappa(\delta) }{\delta} +1.
\]
Specifically, observing that by choosing $A$ sufficiently large and $r_0$ sufficiently small, we have
\[
\begin{split}
L_{\varepsilon} \widetilde{w}- c\widetilde{w}
&= \widetilde{\beta}(1+\widetilde{\beta})a^{nn}(x_n+\varepsilon)^{\alpha}x_n^{\widetilde{\beta}-1}  + (x_n+\varepsilon)^\alpha b^n((1+\widetilde{\beta}) x_n^{\widetilde{\beta}}-1), \\
&\geq  A^{-1}(x_n+\varepsilon)^\beta f \quad \text{in } Q_1^+\cap \{0 < x_n < r_0\}.
\end{split}
\]
The same comparison argument then yields that $|v_\varepsilon(x,t)| \leq \widetilde{\Phi}_{\tau}^+(x,t)$ in $Q_1^+ \cap \{0 < x_n < r_0\}$. Extending this to the entire cylinder $Q_1^+$ as done previously, and taking the infimum over $\tau > 0$, establishes the estimate \eqref{eq:c1 estimate for vva}.

Similarly, noting that $L_{\varepsilon} |x' - y'|^2- c|x' - y'|^2 \leq C(x_n+\varepsilon)^{\alpha}$, the same argument also shows that for every $(y',0) \in\pa B_{1/2}^+$, we have
\[
|v_\varepsilon(x,t)| \leq e^{\Lambda t} \inf_{\tau>0} \left( \kappa(\tau) + \widetilde{M}_{\tau,0}A  x_n  +16|x' - y'|^2  \right),
\]
where $\widetilde{M}_{\tau,0}=\sup_{ \delta> \tau } \frac{\kappa_0(\delta) }{\delta} + 1$. This establishes \eqref{eq:c1 estimate for vva x'}.

\textbf{Step 3.} We show the existence and uniqueness of the solution $v$ to \eqref{eq:equation v0 bcg}, and the convergence $v_\varepsilon \to v$.

Let $\beta=0$, by taking $\tau = x_n^{1/2}$ and noting that $\sup_{\delta > \tau} \frac{\kappa(\delta)}{\delta^{2-\alpha}} \leq x_n^{\frac{\alpha}{2}-1}$, we then obtain
\[
|v_\varepsilon(x,t)| \leq C\left( \kappa(x_n^{1/2}) + A x_n^{1-\frac{\alpha}{2}} \right) \quad \text{in } Q_1^+.
\]
Since $1 < \alpha < 2$, the right-hand side decays to $0$ uniformly as $x_n \to 0^+$. This establishes a uniform oscillation bound for $v_\varepsilon$ near the degenerate boundary $\{x_n=0\}$ that is strictly independent of $\varepsilon$. 
Away from the degenerate boundary $\{x_n=0\}$, the regularized operator $L_{\varepsilon}$ is strictly and uniformly parabolic. Therefore, standard parabolic theory applies, granting the family $\{v_\varepsilon\}$ a uniform modulus of continuity. By the Arzelà-Ascoli theorem, as $\varepsilon \to 0^+$, we can extract a subsequence that converges uniformly in $\overline{Q}_1^+$ to a continuous limit function $v$, which solves \eqref{eq:equation v0 bcg}. By the maximum principle, such a limit is unique, ensuring that the entire family $v_\varepsilon$ converges to $v$. This proves both the existence of the solution and the uniform convergence of the approximation.
\end{proof}

\begin{cor}\label{coro:boundary holder}
Assume all the assumptions in Theorem \ref{thm:Existence and uniqueness}. Let $\sigma^* = 2 - \alpha + \beta$ if $2 - \alpha + \beta < 1$, and $\sigma^* = 1$ if $2 - \alpha + \beta > 1$. 
\begin{enumerate}
\item[i)] Assume that $|g(x,t)| \leq K x_n^{\sigma}$ for some $\sigma \in (0, \sigma^*]$ and $K>1$. Then we have
\[
|v_{\varepsilon}(x,t)| \leq C K x_n^{\sigma} \quad \text{in } Q_1^+.
\]
\item[ii)] Assume that $|g(x,0)| \leq K x_n^{\sigma}$ for some $\sigma \in (0, \sigma^*]$ and $K>1$. Then we have
\[
|v_{\varepsilon}(x,t)| \leq C K x_n^{\sigma} \quad \text{in } Q_1^+ \cap \{|x'| \leq 3/4\}.
\]
\end{enumerate}
\end{cor}

\begin{proof}
Observe that for any $\sigma \leq \sigma^*$, we have
\[
\sup_{\delta > \tau} \delta^{\sigma - \sigma^*} = \tau^{\sigma - \sigma^*}.
\]
Utilizing this property and setting $\tau = x_n$ in \eqref{eq:c0 estimate for vva} when $2 - \alpha + \beta < 1$, and in \eqref{eq:c1 estimate for vva} when $2 - \alpha + \beta \geq 1$, we obtain the first estimate. The second estimate follows analogously from \eqref{eq:c0 estimate for vva x'} and \eqref{eq:c1 estimate for vva x'}. This completes the proof. 
\end{proof}

Next, assuming $\beta>\alpha-1$, we establish the local Lipschitz regularity estimates near the degenerate boundary in a smaller cylinder. The following theorem summarizes the optimal linear growth, the local H\"{o}lder regularity, and the gradient estimates for the solution. While we state these results for the standard nested cylinders $Q_{1/2}^+$, $Q_{3/4}^+$, and $Q_1^+$ to maintain a clean and readable presentation, they can be readily generalized to any concentric sub-cylinders $Q_r^+ \subset Q_R^+$ up to the degenerate boundary via standard scaling and covering arguments.

\begin{thm}\label{thm:local boundary regularity}
Assume all the assumptions in Theorem \ref{thm:Existence and uniqueness}. Let us assume that $\alpha \in (1,2)$ and $\beta \in (\alpha-1,\infty)$.
Then, there exists a constant $C > 0$, depending only on $n, \alpha, \beta$, and $\lambda$, such that
\[
|v_{\varepsilon}(x,t)| \leq C \left( \|g\|_{L^\infty(\partial_p Q_1^+)} + \|f\|_{L^\infty(Q_1^+)} \right) x_n\quad \text{in } Q_{1/2}^+.
\]    
\end{thm}
\begin{proof}
Let us assume \eqref{eq:normalizationg} holds. Note that for any $s\in R$, the function  $\tilde v_{\varepsilon}:=(t-s)v_{\varepsilon}$ satisfies
\[
L_{\varepsilon} \tilde v_{\varepsilon} = (t-s)(x_n+\varepsilon)^\beta f + v_{\varepsilon} \quad \text{in } Q_1^+
\]
with
\[
\tilde{v}_{\varepsilon} = 0 \quad \text{on } \{t = s\},
\]
which permits the application of Corollary \ref{coro:boundary holder}, case ii) for the case of $g(x,0)=0$.

For simplicity, we assume $\beta \in (\alpha-1,1]$. The case $\beta>1$ follows by the same argument, or by replacing $\beta$ with $1$ in the estimates since $x_n^\beta\le x_n$ in $Q_1^+$.
 Let
\[
m = \left\lfloor \frac{\alpha-1}{\beta-\alpha+1} \right\rfloor + 2, \quad \delta = \frac{\alpha-1}{m} \in (0, \beta-\alpha+1),
\]
and define the sequence of exponents
\[
\beta_k = k\delta \leq \alpha-1 <\beta, \quad k = 0, \dots, m.
\]
We proceed by induction on $k = 0, 1, \dots, m-1$ to show that
\[
v_{\varepsilon}(x,t) \leq C_k x_n^{2-\alpha+\beta_k} \quad \text{in } Q_{1/2^{k+1}}^+.
\]

For the base case $k = 0$, applying Corollary \ref{coro:boundary holder}, case ii), to $tv_{\va}$ with parameters $\beta =\beta_0= 0$ and $\sigma_0=2-\alpha$, we obtain from the fact $|t(x_n+\varepsilon)^\beta f + v_{\varepsilon} |\leq 2^2+1$ in $Q_1^+$ that
\[
|tv_{\varepsilon}(x,t)| \leq C_0 x_n^{2-\alpha} \quad \text{in } Q_1^+ \cap \{|x'| \leq 3/4\},
\]
and consequently,
\[
|v_{\varepsilon}(x,t)|\leq C_0 x_n^{2-\alpha} \quad \text{in } Q_{1/2}^+.
\]
Now, suppose the claim holds for some $k$. 
Then we have
\[
v_{\varepsilon}(x,t) \leq C_k x_n^{2-\alpha+\beta_k} \leq C_k x_n^{\beta_{k+1}} \quad \text{in } Q_{1/2^{k+1}}^+,
\]
and consequently,
\[
|L_{\varepsilon}  (t-1+1/2^{k+1}) v_{\varepsilon}| = |(t-1+1/2^{k+1})(x_n+\varepsilon)^\beta f + v_{\varepsilon}|\leq  \widetilde{C}_k (x_n+\varepsilon)^{\beta_{k+1}}  \quad \text{in } Q_{1/2^{k+1}}^+.
\]
By applying Corollary \ref{coro:boundary holder}, case ii) with  with parameters $\beta =\beta_{k+1}$ and $\sigma_k=2-\alpha+\beta_{k+1}$, to a suitable parabolic scaling of $(t-1+1/2^{k+1}) v_{\varepsilon}$ for $Q_{1/2^{k+1}}^+$, we deduce the estimate
\[
|(t-1+1/2^{k+1}) v_{\varepsilon}| \leq C_{k+1} x_n^{2-\alpha+\beta_{k+1}} \quad \text{in }Q_{1/2^{k+2}}^+,
\]
which implies the claim 
for the case $k+1$. This completes the inductive step, yielding
\[
|v_{\varepsilon}(x,t)| \leq C  x_n \quad \text{in } Q_{1/2^{m+1}}^+.
\]
The desired estimate then follows from a standard scaling and covering argument. 
\end{proof}

\begin{thm}\label{prop:local boundary regularity}
Assume all the assumptions in Theorem \ref{thm:local boundary regularity}. Suppose in addition that $a^{ij}\in \C^{\gamma}(\overline{Q}_{3/4}^+)$ for some $\gamma>0$.
Then, there exists a constant  $C_{\gamma}> 0$, depending only on $n, \alpha, \beta$, $\lambda$, $\gamma$ and  the H\"{o}lder norm $\|a_{ij}\|_{\C^{\gamma}(Q_{3/4}^+)}$, such that
\[
\|Dv_{\varepsilon}\|_{L^\infty(Q_{1/2}^+)} \leq C_{\gamma} \left( \|g\|_{L^\infty(Q_1^+)} + \|f\|_{L^\infty(Q_1^+)} \right).
\] 
\end{thm}

\begin{proof}
For simplicity, we assume the normalization \eqref{eq:normalizationg} holds and denote $u = v_{\varepsilon}$ and fix $(x', x_n, t) \in Q_{1/2}^+$.  

\textbf{Case 1:}  $x_n \geq \varepsilon$. Let $r = x_n$ and consider the scaled function $u_r$ defined via \eqref{eq:scale at x_n=0} at $z_0 = (x', 0, t)$. 
The scaled function $u_r$ solves a uniformly parabolic equation \eqref{eq:scaled version equation} in $Q_{3/4}(0,1,1)$ with Hölder continuous coefficients. Thus, by Theorem \ref{thm:local boundary regularity}, we have
\[
\|u_r\|_{L^\infty(Q_{3/4}(0,1,1))} \leq  r^{\alpha-2} \|u\|_{L^\infty(Q_{2r}^+(z_0))} \leq C r^{\alpha-2} x_n \leq C r^{\alpha-1}.
\]
Consequently, applying the classical Lipschitz estimate, we derive that
\[
|Du(x,t)| = r^{1-\alpha} |D u_r(0,1,1)| \leq C r^{1-\alpha} \left( \|u_r\|_{L^\infty(Q_{3/4}(0,1,1))} + \|(x_n^{\beta}f)_r\|_{L^\infty(Q_1^+)} \right) \leq C,
\]
where in the last inequality we used $\|(x_n^{\beta}f)_r\|_{L^\infty(Q_1^+)} \leq C r^{\beta} \leq C r^{\alpha-1}$  for $\beta > \alpha - 1$.

\textbf{Case 2:}  $x_n\leq \varepsilon$. We consider the scaled function $\widetilde{u}=u_{2\varepsilon}$ defined via \eqref{eq:scale at x_n=0} at $z_0 = (x', 0, t)$.
Then $\widetilde{u}$ solves the uniformly parabolic equation
\[
\begin{cases}
\partial_s \widetilde{u}-(y_n+1)^\al (a^{ij}D_{ij}\widetilde{u}+\va b^i D_i\widetilde{u})+\va^{2-\al}c\widetilde{u}=\widetilde{f} & \text{in} \quad Q_1^+, \\
\wt{u} = 0 & \text{on} \quad \{x_n=0\},
\end{cases}
\]
where $\widetilde{f}=\varepsilon^{\beta} (y_n+1)^\beta f(\varepsilon y,\varepsilon^{2-\al}s)$. By Theorem \ref{thm:local boundary regularity}, we have
\[
\|\widetilde{u}\|_{L^\infty(Q_1^+)}\leq \|u_{2\varepsilon}\|_{L^\infty(Q_1^+)} = (2\varepsilon)^{\alpha-2}\|u\|_{L^\infty(Q_{2\varepsilon}^+(z_0))} \leq C\varepsilon^{\alpha-2} x_n\leq C\varepsilon^{\alpha-1}.
\]
Consequently, applying the classical boundary Lipschitz estimate, we derive that
\[
|Du(x,t)|\leq C\varepsilon^{1-\alpha}\|D\widetilde{u}\|_{L^\infty(Q_{1/2}^+)} \leq C \varepsilon^{1-\alpha} \left( \|\widetilde{u}\|_{L^\infty(Q^+_{3/4})} + \|\widetilde{f}\|_{L^\infty(Q_{3/4}^+)} \right) \leq C.
\]
\end{proof}

\begin{thm}\label{thm:holder estimates for general coefficients}
Assume all the assumptions in Theorem \ref{prop:local boundary regularity}. There exist constants $\gamma_0 >0$ and $C > 0$, depending only on $n, \alpha, \beta$, and $\lambda$, such that  $v \in \C^{\gamma_0}(Q_{1/2}^+)$ and satisfies
\[
    \|v\|_{\C^{\gamma_0}(Q_{1/2}^+)} \leq C \left( \|g\|_{L^\infty(\partial_p Q_1^+)} + \|f\|_{L^\infty(Q_1^+)} \right),
    \]
Moreover, if $g \in \C^{\gamma_0}(\partial_p Q_1^+)$, 
then $v \in \C^{\gamma_0}(\overline{Q}_1^+)$ and satisfies
\begin{equation}\label{eq:global holder}
    \|v\|_{\C^{\gamma_0}(Q_1^+)} \leq C \left(\|f\|_{L^\infty(Q_1^+)} + \|g\|_{\C^{\gamma_0}(\partial_p Q_1^+)} \right).
\end{equation}
Furthermore, if we additionally assume that $\|a^{ij}\|_{\C^{\tilde{\gamma}}}\leq \lambda^{-1}$ for some $\tilde{\gamma}> 0$, then the estimates hold for all $\gamma_0 \in (0, \gamma^*)$, where the constant $C$ now also depends on $\tilde{\gamma}$.
\end{thm}

\begin{proof}
We only detail the proof of \eqref{eq:global holder}. The first one follows from the boundary Lipschitz estimates in Theorem \ref{prop:local boundary regularity} combined with an analogous scaling argument below.

For the sake of simplicity, we may normalize the solution such that 
\[
\|v\|_{L^\infty(Q_1^+)} + \|f\|_{L^\infty(Q_1^+)} + \|g\|_{\C^{\gamma_0}(Q_1^+)} \leq 1.
\]
Let $\gamma_0 \in (0,1)$ be the universal exponent defined in Theorem \ref{thm:standard_interior_boundary_estimates}. Invoking Lemma \ref{lem:euclidean_patching}, we can bound the global intrinsic H\"{o}lder seminorm as follows:
\[ 
[v]_{\C^{\gamma_0}(Q_{1}^+)} \le C(n, \alpha, \gamma_0) \left( [v]_{\C_{bd}^{\gamma_0}(Q_{1}^+)} + [v]_{\C_{int}^{\gamma_0}(Q_{1}^+)} \right). 
\]
Note that $\|g\|_{\C^{\gamma_0}(Q_1^+)} \leq 1$ implies $|g(x,t)| \leq C x_n^{\gamma_0(1-\frac{\alpha}{2})}$.
Corollary \ref{coro:boundary holder} i) then implies the pointwise bound 
\[
|v(x,t)| \leq C x_n^{\gamma_0(1-\frac{\alpha}{2})}.
\]
Since $v=0$ on the degenerate boundary $\{x_n = 0\}$, this yields $[v]_{\C_{bd}^{\gamma_0}(Q_{1}^+)} \leq C$.

Next, concerning the interior term $[v]_{\C_{int}^{\gamma_0}(Q_{1}^+)}$, we recall the boundary-centered scaling $v_{r,z_0}$ defined in \eqref{eq:scale at x_n=0} alongside the structural inequality \eqref{eq:seminorm_scaling}
\[
[v]_{\C_{int}^{\gamma_0}(Q_1^+)}=\sup_{\substack{z_0 \in \partial_p Q_1^+ \cap \{x_n = 0\} \\ 0 < r < \rho }} r^{(2-\gamma_0)(1-\frac{\alpha}{2})} [v_{r,z_0}]_{C^{\gamma_0}(Q_{1/2}(0,1,1))}.
\]
Specifically, noting that
\[
\|v_{r,z_0}\|_{L^\infty(Q_{1/2}(0,1,1))}\leq r^{\alpha-2}\|v\|_{L^\infty(Q_{2r}^+(z_0))} \leq Cr^{-(2-\gamma_0)(1-\frac{\alpha}{2})},
\]
the desired interior bound follows by applying \eqref{eq:classical holder} uniformly to each rescaled $v_{r,z_0}$.

Furthermore, if we additionally assume that $\|a^{ij}\|_{\C^{\tilde{\gamma}}}\leq \lambda^{-1}$ for some $\tilde{\gamma}> 0$, then the above estimates hold for all $\gamma_0 \in (0, \gamma^*)$, where the constant $C$ now also depends on $\tilde{\gamma}$.
\end{proof}

\section{Schauder regularity for linear equations on half cylinder}
\label{sec:schauder regularity on halc cylinder}

\subsection{Higher-order local derivative estimates}

In this section, we investigate the boundary Schauder regularity for the linear degenerate parabolic equation
\begin{equation}\label{eq:equation_v0}
    \begin{cases}
        L_0 v := v_t - x_n^\alpha \Delta v = x_n^\beta f & \text{in } Q_1^+, \\
        v = 0 & \text{on } \partial_p Q_1^+ \cap \{x_n = 0\}.
    \end{cases}
\end{equation}
Our primary focus is to quantitatively capture the behavior of solutions near the degenerate boundary $\{x_n = 0\}$. 

In Theorem~\ref{prop:local boundary regularity}, we obtain the Lipschitz bound
\[
    \|Dv\|_{L^\infty(Q_{1/2}^+)} \leq C \left( \|v\|_{L^\infty(Q_1^+)} + \|f\|_{L^\infty(Q_1^+)} \right).
\]
In this subsection, we establish the higher-order derivative estimates for $v$, assuming sufficient regularity on the source term $f$.

\begin{thm}\label{thm:C2 estimate}
Let $v$ be a solution to \eqref{eq:equation_v0}, and suppose that the source term satisfies $f \in C^3(Q_1^+)$. Then $v \in C^1(\overline{Q}_{1/2}^+)$, and there exists a constant $C > 0$, depending only on $n, \alpha$, and $\beta$, such that
\begin{equation}\label{eq:f_C3_estimates}
\begin{aligned}
\sum_{\substack{0 \le l+|k|+m \le 2 \\ 0 \le m \le 1}} \|\partial_t^l D_{x'}^k D_{x_n}^m v\|_{L^\infty(Q_{1/2}^+)}&+\|x_n^{\al-\beta}D_{x_nx_n}v\|_{L^\infty(Q_{1/2}^+)}\\
&\qquad\leq C \left( \|v\|_{L^\infty(Q_1^+)} + \|f\|_{C^3(Q_1^+)} \right).  
\end{aligned}
\end{equation}
Moreover, we have $\pa_{tt} v=0$ on $\pa_p Q_{1/2}^+ \cap \{x_n=0\}$ and
\[
\|D(\pa_t^2 v)\|_{L^\infty(Q_{1/2}^+)}\leq C  \left( \|v\|_{L^\infty(Q_1^+)} + \|f\|_{C^3(Q_1^+)} \right).
\]
\end{thm}

As a corollary of Theorem \ref{thm:C2 estimate}, we obtain

\begin{prop}\label{lem:approximation lem} 
Let $f_0$ be a constant and $v$ be a solution of     
\[
\begin{cases}
        L_0 v := v_t - x_n^\alpha \Delta v = x_n^\beta f_0 & \text{in } Q_1^+, \\
        v = 0 & \text{on } \pa_p Q_1^+ \cap \{x_n = 0\}.
    \end{cases}
\]
Suppose $\|v\|_{L^{\infty}(Q_1^+)} +|f_0|\leq 1$, then for every $r\leq 1/2$, we have
\[
\|v-P_{(0,1)}\|_{L^\infty(Q_r^+)}\leq Cr^{3-\al},
\]  
where  $C$ is a constant depending only on $n,\al$ and $\beta$, and
\[
P_{(0,1)}(x,t)= v_n(0,1) x_n-\frac{f_0}{(2-\al+\beta)(1-\al+\beta)}x_n^{2-\al+\beta}.
\]   
\end{prop}
\begin{proof} 
Let $h=v-P_{(0,1)}$, then
\[
L_0h=0 \quad \text{in}~Q_1^+, \qquad  h=0 \quad \text{on}~ \pa_p Q_1^+\cap \{x_n=0\}.
\]
The local derivative estimates in Theorem \ref{thm:C2 estimate} yield bounds on the derivatives of $h$ in $Q_r^+$. More precisely, for $i,j=1,\dots,n-1$, we have
\begin{align*}
&|h_{x_ix_j}|,\ |h_{x_ix_n}|,\ |h_{x_it}|,\ |h_{x_nt}|\le C\|h\|_{L^\infty(Q_1^+)},\\
&|h_{x_nx_n}|\le Cx_n^{1-\al}\|h\|_{L^\infty(Q_1^+)},\quad |h_{tt}|\le Cx_n\|h\|_{L^\infty(Q_1^+)}.
\end{align*}
Combining these bounds with $|x_i|\le r$, $0\le x_n\le r$, and $0\le 1-t\le r^{2-\al}$, we obtain
\[
\|v-P_{(0,1)}\|_{L^\infty(Q_r^+)}=\|h\|_{L^\infty(Q_r^+)}\le Cr^{3-\al}\|h\|_{L^\infty(Q_1^+)} \leq Cr^{3-\al}.
\]
\end{proof}

To circumvent the lack of \emph{a priori} regularity of $v$ for direct differentiation, we introduce a parameter $\varepsilon > 0$ and consider the classical solution to the strictly parabolic approximation:
\begin{equation}\label{eq:equation}
\begin{cases}
L_{0,\varepsilon} v_\varepsilon := (v_\varepsilon)_t - (x_n+\varepsilon)^\alpha \Delta v_\varepsilon = (x_n+\varepsilon)^\beta f & \text{in} \quad Q_1^+, \\
v_\varepsilon = v & \text{on} \quad \partial_p Q_1^+.
\end{cases}
\end{equation}
For any fixed $\varepsilon > 0$, the operator $L_{0,\varepsilon}$ is non-degenerate even at $\{x_n = 0\}$; hence, classical parabolic theory guarantees that $v_\varepsilon$ is sufficiently smooth up to the boundary, provided that $f$ is sufficiently smooth.
Once $\varepsilon$-uniform higher-order estimates are established for $v_\varepsilon$, we recover the regularity of the original solution $v$ by passing to the limit as $\varepsilon \to 0$.  

\begin{lem}\label{lem:bound of vii}
Suppose that $u$ is a solution of \eqref{eq:equation} satisfying $u=0$ on $\pa_p Q_1^+\cap \{x_n=0\}$, and that
\[
|u|\leq 1,\quad |f|\leq 1,\quad |f_{x_j}|\le 1\quad  \text{in}~Q_1^+,
\]
for some $1\leq j\le n-1$. Then
\[
|u_{x_ix_j}|\le C \quad \text{in } Q_{1/2}^+.
\]
for some constant $C$ depending only on $n,\al$ and $\beta$.
\end{lem}

\begin{proof}
Since $u$ is smooth up to the boundary, the homogeneous condition $u = 0$ on $\{x_n=0\}$ implies that $u_{x_j} = 0$ on $\{x_n=0\}$ for each $1 \leq j \leq n-1$. Furthermore, the tangential derivative $u_{x_j}$ satisfies the differentiated equation
\[
L_{0,\varepsilon} u_{x_j} = (x_n+\varepsilon)^\beta f_{x_j} \quad \text{in } Q_{3/4}^+.
\]
By Theorem \ref{prop:local boundary regularity}, we already have $|Du| \leq C$ in $Q_{3/4}^+$. A subsequent application of Theorem \ref{prop:local boundary regularity} to $u_{x_j}$ yields the second-order estimate $|D(u_{x_j})| \leq C$ in $Q_{1/2}^+$.
\end{proof}

\begin{lem}\label{lem:bound of vt}
Assume all the assumptions in Lemma \ref{lem:bound of vii}. Assume in addition that
\[
|u_{x_i}|\leq 1,
\quad |u_{x_ix_i}|\le 1,
\quad (\forall~i=1,\dots,n-1),
\quad |u_{x_n}|\leq 1
\quad  \text{and}\quad
|f_t|\leq 1  \quad \text{in}~Q_1^+.
\]
Then $|u_t|\le C$ in $Q_{1/2}^+$. for some constant $C$ depending only on $n,\al$ and $\beta$. 
Moreover, $|u_t|\leq Cx_n$.
\end{lem}
\begin{proof}

By assumption, we have $|u_{x_n}|\le 1$. If let $\wt{u}=u+2x_n$, then $\wt{u}$ satisfies
\[
\wt{u}_t=(x_n+\varepsilon)^\al \Delta \wt{u}+(x_n+\varepsilon)^\beta f,\quad \wt{u}_{x_n}=u_{x_n}+2\in [1,3] \quad \text{and}\quad \wt{u}_t=u_t.
\]
Thus, if we establish the conclusion for $\wt{u}$,  it also holds for $u$. Therefore, without loss of generality, we may assume that
\[
1\leq u_{x_n}\leq 3.
\]
Let $A$ and $P$ be two large numbers, $\eta$ be a smooth function satisfying 
\[
\eta=0 \quad \text{on}~ \pa_{p} Q_{3/4}^+\cap \{x_n>0\}\quad \text{and}\quad  \eta>0 \quad \text{in}~ Q_{1/2}^+, 
\]
and define
\[
w=A(P-(x_n+\varepsilon)^{\frac{\al}{2}}u_{x_n})^2+\eta^2 u_t^2.
\]
Since 
\[
L_{0,\va}u_t=(x_n+\va)^\beta f_t \leq C,
\]
we have 
\begin{align*}
L_{0,\varepsilon}\left(\eta^2 u_{t}^2\right) 
& =\eta^2L_{0,\varepsilon}u_{t}^2 +u_{t}^2 L_{0,\varepsilon} \eta^2 -2(x_n+\va)^\al  \langle \nabla \eta^2,  \nabla  u_{t}^2\rangle\\
& \leq \eta^2 L_{0,\varepsilon}u_{t}^2+C \left(u_{t}^2+ (x_n+\va)^\al \eta |u_{t}|  |\nabla u_{t}|\right) \\
& \leq 2u_{t} L_{0,\varepsilon} u_{t}-2(x_n+\va)^\al|\nabla u_{t}|^2+C \left(u_{t}^2+ (x_n+\va)^\al \eta |u_{t}|  |\nabla u_{t}|\right) \\
& \leq C |u_{t}| -2(x_n+\va)^\al|\nabla u_{t}|^2+C \left(u_{t}^2+ (x_n+\va)^\al \eta |u_{t}|  |\nabla u_{t}|\right) \\
& \leq C(u_{t}^2+|u_{t}|+1).
\end{align*}
Notice that
\begin{align*}
&L_{0,\va} (P-(x_n+\va)^{\frac{\al}{2}}u_{x_n})\\
&=-(x_n+\va)^{\frac{\al}{2}}u_{x_nt}+(x_n+\va)^\al \Delta \bigl((x_n+\va)^{\frac{\al}{2}}u_{x_n}\bigr)\\
&= -(x_n+\va)^{\frac{\al}{2}}u_{x_nt}+(x_n+\va)^{\frac{3\al}{2}}\Delta u_{x_n}+  \frac{\al}{2}(\frac{\al}{2}-1) (x_n+\va)^{\frac{3\al}{2}-2} u_{x_n}\\
& \quad +2(x_n+\va)^\al \langle \nabla (x_n+\va)^\frac{\al}{2}, \nabla u_{x_n}\rangle\\
&=-(x_n+\va)^{\frac{\al}{2}} \Bigl(\al (x_n+\va)^{\al-1}\Delta u+(x_n+\va)^\beta f_{x_n}+\beta (x_n+\va)^{\beta -1}f \Bigr)\\
&\quad +\frac{\al}{2}(\frac{\al}{2}-1) (x_n+\va)^{\frac{3\al}{2}-2} u_{x_n}+\al (x_n+\va)^{\frac{3\al}{2}-1} u_{x_nx_n}\\
&=-\al (x_n+\va)^{\frac{3\al}{2}-1}\Delta_{x'} u-(x_n+\va)^{\frac{\al}{2}+\beta}f_{x_n}-\beta (x_n+\va)^{\frac{\al}{2}+\beta-1} f\\
&\quad -\frac{1}{4}\al(2-\al)(x_n+\va)^{\frac{3\al}{2}-2} u_{x_n}\\
& = -(x_n+\va)^{\frac{3\al}{2}-2} \Bigl(\al (x_n+\va)\Delta_{x'}u+(x_n+\va)^{2+\beta-\al}f_{x_n}\\
&\quad +\beta (x_n+\va)^{\beta+1-\al}f+\frac{1}{4}\al(2-\al) u_{x_n}  \Bigr),
\end{align*}
then when $x_n$ is small enough, we have 
\[
L_{0,\va} (P-(x_n+\va)^{\frac{\al}{2}}u_{x_n})\leq 0,
\] 
otherwise,
\[
L_{0,\va} (P-(x_n+\va)^{\frac{\al}{2}}u_{x_n})\leq C,
\]
for some positive constant $C$. Besides,
\[
\begin{split}
(x_n+\va)^\al|\nabla \bigl((x_n+\va)^{\frac{\al}{2}}u_{x_n}\bigr)|^2 &\geq (x_n+\va)^{2\al}|u_{x_nx_n}|^2\\
& \geq \Bigl| u_t-(x_n+\va)^\al \Delta_{x'} u-(x_n+\va)^\beta f\Bigr|^2\\
&\geq |u_t|^2-C |u_t|.
\end{split}
\]
Therefore, if choose $P$ large enough such that 
\[
\frac{P}{2}\leq P-(x_n+\varepsilon)^{\al/2}u_{x_n}\leq  \frac{3P}{2},
\]
we have 
\[
\begin{split}
L_{0,\va} \bigl( (P-(x_n+\va)^{\frac{\al}{2}}u_{x_n})^2\bigr)&=2(P-(x_n+\va)^{\frac{\al}{2}}u_{x_n})L_{0,\va} (P-(x_n+\va)^{\frac{\al}{2}}u_{x_n})\\
&\quad-2(x_n+\va)^\al |\nabla (-(x_n+\va)^{\frac{\al}{2}}u_{x_n})|^2\\
&\leq -2|u_t|^2+C|u_t|+CP.
\end{split}
\]
By choosing $A$ large enough, we have 
\[
L_{0,\va} w\leq A(-2|u_t^2|+C|u_t|+C)+C(u_t^2+|u_t|+1)\leq C \quad \text{in}~\{\eta>0\}.
\]
Combining with the boundary value that
\[
A(P-(x_n+\varepsilon)^{\frac{\al}{2}}u_{x_n})^2+\eta^2 u_t^2=A (P-(x_n+\varepsilon)^{\frac{\al}{2}}u_{x_n})^2\leq C,
\]
the maximum principle implies that
\[
w\leq C \quad \text{in}~\{\eta>0\},
\]
which is 
\[
|u_t|\leq \frac{\sqrt{w}}{\eta}\leq C(n,\al,\beta)\quad \text{in}~Q_{1/2}^+.
\]
In particular, $u_t$ satisfies
\[
(u_t)_t-(x_n+\varepsilon)^\al \Delta u_t
=(x_n+\varepsilon)^\beta f_t,
\]
and $u_t=0$ on $\pa_p Q_1^+\cap \{x_n=0\}$.
Combining the bound $|u_t|\le C$, the assumption $|f_t|\le 1$, and the proof of Theorem~\ref{thm:Existence and uniqueness}, we obtain
\[
|u_t|\le C x_n.
\]
\end{proof}

Notice that $(x_n+\varepsilon)^\al u_{x_nx_n}$ and $u_t$ satisfy the following two equations:
\[
(x_n+\varepsilon)^\al u_{x_nx_n}=
u_t-(x_n+\varepsilon)^\al\Delta_{x'}u
-(x_n+\varepsilon)^\beta f,
\]
and
\[
(u_t)_t=(x_n+\varepsilon)^\al\Delta u_{t}
+(x_n+\varepsilon)^\beta f_t.
\]
\begin{cor}\label{cor:vt_higher}
Let $u$ be as in Lemma~\ref{lem:bound of vt}.
\begin{enumerate}
\item[(i)] Then
\[
|u_{x_nx_n}|\leq C (x_n+\varepsilon)^{\beta-\alpha}
\quad \text{in } Q_{1/2}^+.
\]
In particular, if $f\equiv 0$, then
\[
|u_{x_nx_n}| \le C (x_n+\varepsilon)^{1-\al}
\quad \text{in } Q_{1/2}^+.
\]
\item[(ii)]If we assume in addition that
\[
|f_{tt}|\le 1,\qquad
|\nabla f_t|\le 1,\qquad
|f_{tx_ix_i}|\le 1 \quad
(\forall~ i=1,\dots,n-1) \quad
\text{in } Q_1^+,
\]
then
\[
|u_{tt}|\le C \quad \text{in } Q_{1/2}^+.
\]
\end{enumerate}
The constant $C$ above depends only on $n,\al$ and $\beta$.
\end{cor}
\begin{proof}
Combining the equation
\[
(x_n+\varepsilon)^\al u_{x_nx_n}
=u_t-(x_n+\varepsilon)^\al\Delta_{x'}u
-(x_n+\varepsilon)^\beta f,
\]
with $|u_t|\leq C(x_n+\varepsilon)$,  $|\Delta_{x'}u| \leq C$, we obtain the result.
If $f\equiv 0$, the result follows from the same reason.

For the equation
\[
(u_t)_t-(x_n+\varepsilon)^\al\Delta u_{t}
=(x_n+\varepsilon)^\beta f_t,
\]
we already know that $|u_t|\leq C$, $|f_t|\leq 1$. Then by the proof of Theorem \ref{thm:Existence and uniqueness}, we have $|u_t|\leq M (x_n+\varepsilon)$. By the proof of Lemma \ref{lem:bound of vt}, if $|f_{tt}|< 1$, $|\nabla f_t|\leq 1$ and $|f_{tx_ix_i}|\leq 1$, we obtain $|u_{tt}|<C$.
\end{proof}

\begin{proof}[\textbf{Proof of Theorem \ref{thm:C2 estimate}}]
By Lemma \ref{lem:bound of vii}, Lemma \ref{lem:bound of vt} and Corollary \ref{cor:vt_higher},
we have 
\begin{align*}
&|(Dv_{\varepsilon})_{x_i}|\leq C, \quad \text{for}~i=1,\dots,n-1.\\
& |(v_{\varepsilon})_t|\leq Cx_n,
\quad |(v_{\varepsilon})_{tt}|\leq C,
\quad |(v_{\varepsilon})_{x_nx_n}|\leq C(x_n+\varepsilon)^{\beta-\alpha}.
\end{align*}
Taking $\va\to 0$,  \eqref{eq:f_C3_estimates} follows.
Moreover, we obtain that for  $1\leq i\leq n-1$,  
\begin{align*}
&|(v_\va)_{x_n}(x+\tau e_i,t)-(v_\va)_{x_n}(x,t)|
\leq  \int_{0}^{\tau} \pa_{x_i} (v_\va)_{x_n}(x+se_i,t) ds
\leq  C \tau, \\
&|(v_\va)_{x_n}(x+\tau e_n,t)-(v_\va)_{x_n}(x,t)|
\leq  \int_{0}^{\tau} \pa_{x_n} (v_\va)_{x_n}(x+se_n,t) ds
\leq  C \tau^{\beta-\al+1}, \\
&|(v_\va)_{x_n}(x,t+\tau)-(v_\va)_{x_n}(x,t)|
\leq  \int_{0}^{\tau} \pa_{t} (v_\va)_{x_n}(x,t+ s ) ds
\leq  C \tau.
\end{align*}
This implies that $(v_\varepsilon)_{x_n} \in C(\overline{Q}_{1/2}^+)$. The uniform continuity of the derivatives in the remaining tangential and temporal directions can be established in an analogous manner. Furthermore, since the $L^\infty$ norms and the moduli of continuity of both $Dv_\varepsilon$ and $(v_\varepsilon)_t$ are uniform with respect to $\varepsilon$, passing to the limit yields that the original solution $v$ belongs to $C^1(\overline{Q}_{1/2}^+)$.
\end{proof}

\subsection{Local estimates for variable coefficients via perturbation}\label{sub:variable_coefficients}

In this section, by employing standard perturbation arguments, we extend the local regularity estimates to the variable-coefficient setting. Specifically, we consider the degenerate boundary value problem
\begin{equation}\label{eq:equation_v0 variable}
\begin{cases}
v_t-x_n^\al \left(a^{ij}D_{ij}v+b^i D_i v\right)+cv=x_n^\beta f & \quad \text{in } Q_1^+,
 \\
v = 0 & \quad \text{on } \partial_p Q_1^+ \cap \{x_n = 0\},
\end{cases}
\end{equation}
where we assume that the variable coefficients $a^{ij}, b^i, c$, along with the source term $f$, belong to the H\"older space $\C^\gamma(Q_1^+)$ and satisfy the uniform structure conditions \eqref{eq:structure_conditions}.

\begin{lem}[Approximation Lemma]\label{lem:solution approximation lem}
Let $0<\gamma<\gamma^*$, and let $v$ be a solution of \eqref{eq:equation_v0 variable}. Suppose that
\be\label{eq:assumption of approximation lemma}
\left\|a^{ij}-\ol{a}\delta_{ij}\right\|_{L^\infty(Q_1^+)}
+\left\|b^i\right\|_{L^\infty(Q_1^+)}
+\left\|c\right\|_{L^\infty(Q_1^+)}
\leq \varepsilon
\ee
for some $\varepsilon\in(0,1)$ and some constant $\ol{a}$ satisfying $\lambda\leq \ol{a}\leq \lambda^{-1}$. Let $h\in C^\infty(Q_{3/4}^+)\cap C(\ol{Q}_{3/4}^+)$ be the solution of
\[
\left\{
\begin{array}{ll}
h_t-x_n^\al \ol{a}\Delta h=x_n^\beta f(0,1) & \text{in } Q_{3/4}^+,\\
h=v & \text{on } \pa_p Q_{3/4}^+.
\end{array}
\right.
\]
then
\[
\|v-h\|_{L^\infty(Q_{1/2}^+)}\leq C \left( \varepsilon^\theta+\|f\|_{L^\infty(Q_1^+)}\right),
\]
for some constants $\theta\in (0,1)$ and $C>0$ depending only on $n,\al,\beta,\gamma,\lambda$.
\end{lem}

\begin{proof}
The existence and uniqueness of $h$ follow directly from Theorem \ref{thm:Existence and uniqueness}. Furthermore, in Section \ref{sec: boundary holder regularity}, the following estimates have been established:
\[
\left\|v\right\|_{\C^{\gamma}(Q_{3/4}^+)}\leq C \left(1+\left\|f\right\|_{L^\infty(Q_{3/4}^+)}\right),\quad |h(x,t)|\leq C \left(1+\left\|f\right\|_{L^\infty(Q_{3/4}^+)}\right)x_n^\beta 
\]
and
\[
\left\|h\right\|_{\C^{\gamma}(Q_{3/4}^+)}\leq C\left(\left\|h\right\|_{L^\infty(Q_{3/4}^+)}+\left\|v\right\|_{\C^\gamma(Q_{3/4}^+)}+\left\|f\right\|_{L^\infty(Q_{3/4}^+)}\right)\leq C\left(1+\left\|f\right\|_{L^\infty(Q_1^+)}\right).
\] 
Since $v - h = 0$ on $\partial_{p} Q_{3/4}^+$, for any small constant $\delta \in (0, 1/4)$,  we have
\[
\|v-h\|_{L^\infty(\partial_{p} Q_{3/4-\delta}^+)}\leq \left\|v-h\right\|_{\C^\gamma(\ol{Q}_{3/4}^+)}\delta^\gamma\leq C\delta^\gamma \left(1+\left\|f\right\|_{L^\infty(Q_1^+)}\right).
\]
Moreover, for all $(\ol{x},\ol{t})\in Q_{3/4-\delta}^+$,  we consider 
the rescaled function
\begin{align*}
&k(x,t)=\frac{h(\delta x+\ol{x},\delta^{2-\al}(t-1)+\ol{t})-h(\ol{x},\ol{t})}{\delta^{(1-\frac{\al}{2})\gamma}}, \quad  \text{if } \ol{x}_n\leq \frac{\delta}{2},\\
&k(x,t)=\frac{h(\delta x+\ol{x},\delta^{2}\ol{x}_n^{-\al}(t-1)+\ol{t})-h(\ol{x},\ol{t})}{\delta^{(1-\frac{\al}{2})\gamma}},\quad   \text{if }\ol{x}_n\geq \frac{\delta}{2}, 
\end{align*}
then, $k$ satisfies
\[
k_t - \overline{a} \left(x_n + \frac{\overline{x}_n}{\delta}\right)^\alpha \Delta k = 
\delta^{2-\alpha-(1-\frac{\alpha}{2})\gamma+\beta} (x_n+\frac{\ol{x}_n}{\delta} )^\beta f(0,1) \quad \text{in } Q_1^+~
\]
when $\overline{x}_n \leq \frac{\delta}{2}$ and
\[
k_t - \overline{a}\left( \frac{\delta}{\ol{x}_n}x_n + 1\right)^\alpha \Delta k = 
\delta^{2-(1-\frac{\alpha}{2})\gamma} \ol{x}_n^{\beta-\al} (\frac{\delta}{\ol{x}_n}x_n + 1)^\beta f(0,1) \quad \text{in } Q_1
\]
when $\overline{x}_n \geq \frac{\delta}{2}$.
Notice that
\[
\|k\|_{L^\infty}
\leq C\|h\|_{\mathscr{C}^\gamma(Q_{3/4}^+)}
\leq C\bigl(1+\|f\|_{L^\infty(Q_1^+)}\bigr).
\]
Combining this estimate with Theorem \ref{thm:C2 estimate} in the case $\overline{x}_n\leq \frac{\delta}{2},$
and with the classical parabolic estimates in the case $\overline{x}_n\geq \frac{\delta}{2},$
then applying a standard scaling argument, we obtain the following weighted bounds for the derivatives of $h$ at $(\overline{x},\overline{t})$,
\[
|\ol{x}_n^{\al-\beta}D^2 h(\ol{x},\ol{t})|\leq C \left( 1+\left\|f\right\|_{L^\infty(Q_1^+)}\right) \delta^{(1-\frac{\al}{2})\gamma-2-\beta+\al}.
\]
Similarly, by Theorem \ref{prop:local boundary regularity} and the scaling argument, we obtain that
\[
|Dh(\ol{x},\ol{t})|\leq C \left( 1+\left\|f\right\|_{L^\infty(Q_1^+)}\right) \delta^{(1-\frac{\al}{2})\gamma-1}.
\]
Let $u=v-h$, then $u$ satisfies that
\[
u_t-x_n^\al \left(a^{ij}D_{ij}u+b^i D_i u\right)+cu =x_n^\beta \wt{f},
\]
where 
\[
\wt{f}= f-f(0,1)+x_n^{\al-\beta}(a^{ij}-\ol{a}\delta_{ij})D_{ij}h+x_n^{\al-\beta} b^i D_ih-ch x_n^{-\beta}.
\]
By previous estimates and \eqref{eq:assumption of approximation lemma}, it follows that
\[
\|\widetilde{f}\|_{L^\infty(Q_{3/4-\delta}^+)} \leq 2\|f\|_{L^{\infty}(Q_1^+)}+\varepsilon \left( 1+\left\|f\right\|_{L^\infty(Q_1^+)}\right) \delta^{(1-\frac{\alpha}{2})\gamma-2-\beta+\alpha}.
\]
Thus, by the Maximum Principle,  if we choose $\delta=\varepsilon^{\frac{2}{4-2\al+\al\gamma+2\beta}}$ and $\theta=\frac{2\gamma}{4-2\al+\al\gamma+2\beta}$, we have
\[
\left\|v-h\right\|_{L^\infty(Q_{3/4-\delta}^+)}
\leq  C \left(\left\|v-h\right\|_{L^\infty(\pa_p Q_{3/4-\delta}^+)}
+\|\wt{f}\|_{L^\infty(Q_{3/4-\delta}^+)}\right)
\leq  C\left(\varepsilon^\theta+\left\|f\right\|_{L^\infty(Q_1^+)}\right).
\]
\end{proof}

\begin{thm}\label{thm:polynomial approximation}
Let $0<\gamma<\gamma^*$,  and let $v$ be a solution to \eqref{eq:equation_v0 variable}. We further assume
\[
\left\|a^{ij}\right\|_{\C^\gamma(Q_{1}^+)}+\|b^i\|_{\C^\gamma(Q_{1}^+)}+\|c\|_{\C^\gamma(Q_1^+)}+\|f\|_{\C^\gamma(Q_1^+)}\leq \frac{1}{\lambda}.
\]  
Then, for each $z_0=(x_0,t_0)\in \partial_p Q_{1/2}^+\cap\{x_n=0\}$, there exists a polynomial of the form
\[
P_{z_0}(x,t)=a(z_0)x_n-\frac{f(z_0)}{a^{nn}(z_0)(2-\al+\beta)(1-\al+\beta)}x_n^{2-\al+\beta},
\]
where $a(z_0)$ is a constant satisfying $|a(z_0)|\leq C$, which will later be identified with the normal derivative $v_n(z_0)$, such that
\begin{equation}\label{eq:approximated polynomial}
|v(x,t)-P_{z_0}(x,t)|\leq C\left(s[z,z_0]^{\frac{2}{2-\al}}\right)^{2-\al+\beta+(1-\frac{\al}{2})\gamma}\quad \text{in } Q_{1/4}^+(z_0).
\end{equation}
Here $C>0$ depends only on $n$, $\alpha$, $\beta$, $\gamma$, and $\lambda$.
\end{thm}

\begin{proof}
Up to a translation, we may assume without loss of generality that $z_0 = (0,1)$. Moreover, by replacing $v$ with a suitably scaled function $v_{r_0}$ centered at $(0,1)$ for a sufficiently small radius $r_0 > 0$—which we still denote by $v$ for simplicity—we may assume that the coefficients satisfy 
\[
[a^{ij}]_{\C^\gamma(Q_1^+)} + \|b^i\|_{\C^\gamma(Q_1^+)} + \|c\|_{\C^\gamma(Q_1^+)} \leq \varepsilon,
\]
for a sufficiently small constant $\varepsilon > 0$.
Furthermore, without loss of generality, we may also assume that
\[
\|v\|_{L^\infty(Q_1^+)} \leq 1 \quad \text{and} \quad \|f\|_{\C^\gamma(Q_1^+)} \leq \varepsilon.
\]

We first simplify the coefficient matrix. Set
\[
A_{z_0}=\{a^{ij}(z_0)\}_{n\times n}.
\]
Since $A_{z_0}$ is symmetric and uniformly positive definite, there exists a sliding transformation 
\[
\wt{x}=S_{z_0}x
\]
where $S_{z_0}$ is an upper triangular matrix, such that
\[
S_{z_0} A_{z_0} S_{z_0}^T=I.
\]
More precisely, $S_{z_0}$ can be chosen so that $y_n=x_n$. Hence, the flat boundary $\{x_n=0\}$ is mapped onto $\{y_n=0\}$, and the degeneracy factor $x_n^\alpha$ is transformed into $y_n^\alpha$. Therefore, after this linear change of variables, we may assume that
\[
a^{ij}(z_0)=a^{nn}(z_0)\delta_{ij}.
\]

Note that the above simplification does not alter the form of \eqref{eq:approximated polynomial}.
Next, we show that there exist $0<r<1/2$ and a sequence of polynomials
\[
P_k(x,t)=a_k\cdot x_n-\frac{f(0,1)}{a^{nn}(0,1)(2-\al+\beta)(1-\al+\beta)}x_n^{2-\al+\beta}
\]
such that for every $k\ge 0$,
\[
\|v-P_k\|_{L^\infty(Q_{r^k}^+)}\leq  r^{k\left(2-\al+\beta+\left(1-\frac{\al}{2}\right)\gamma\right)},
\]
and for every $k\ge 1$,
\[
r^{k-1}|a_k-a_{k-1}|\leq Cr^{(k-1)\left(2-\al+\beta+\left(1-\frac{\al}{2}\right)\gamma\right)},
\]
where we set $a_{-1}=a_0= 0$.

The case $k=0$ is immediate. Assume that the above estimates hold for $k=0,1,\dots,l$, we show that they also hold for $k=l+1$.
Consider the function
\[
\wt{v}(x,t)=\frac{(v-P_l)(r^l x,r^{(2-\al)l}(t-1)+1)}{r^{l(2-\al+\beta+(1-\frac{\al}{2})\gamma)}}.
\]
Then $\left\| \wt{v}\right\|_{L^\infty(Q_1^+)}\leq 1$ and it satisfies that
\[
\wt{v}_t-x_n^\al \left(\wt{a}^{ij}D_{ij}\wt{v}+\wt{b}^iD_i \wt{v}\right)+\wt{c}\wt{v}=x_n^\beta \wt{f},
\]
where $\wt{a}^{ij}, \wt{b}^i$, $\wt{c}$ are defined similar to \eqref{eq:scaled coefficients} and 
\begin{align*}
&\wt{f}(x,t)=\frac{1}{a^{nn}(0,1)r^{l(1-\frac{\al}{2})\gamma}}\left(a^{nn}(0,1)f(r^l x,r^{(2-\al)l}(t-1)+1)-\wt{a}^{nn}f(0,1)\right)\\
&+r^{l(\al-\beta-(1-\frac{\al}{2})\gamma)}x_n^{\al-\beta} b^{n}(r^l x,r^{(2-\al)l}(t-1)+1) (a_l-\frac{f(0,1)}{a^{nn}(0,1)(1-\al+\beta)}(r^l x_n)^{1-\al+\beta})\\
&- c(a_l r^{l(1-\beta-(1-\frac{\al}{2})\gamma)}x_n^{1-\beta}-\frac{f(0,1)}{a^{nn}(0,1)(1-\al+\beta)(2-\al+\beta)}r^{l(2-\al-(1-\frac{\al}{2})\gamma)}x_n^{2-\al})
\end{align*}
satisfying $\wt{f}(0,1)=0$. Since
\[
\left\|f(r^l\cdot, r^{(2-\al)l}(\cdot-1)+1)-f(0,1)\right\|_{L^{\infty}(Q_1^+)}  \leq \left\|f\right\|_{\C^\gamma(Q_1^+)} r^{l(1-\frac{\alpha}{2})\gamma}\leq  \varepsilon r^{l(1-\frac{\alpha}{2})\gamma},
\]
and 
\[
\left\|\widetilde{a}^{ij}-a^{nn}(0,1)\delta_{ij}\right\|_{L^{\infty}(Q_1^+)} \leq [a^{ij}]_{\C^\gamma(Q_1^+)} r^{l(1-\frac{\alpha}{2})\gamma}\leq  \varepsilon r^{l(1-\frac{\alpha}{2})\gamma}  ,
\]
we have
\[
\|\wt{f}\|_{L^\infty(Q_1^+)}\leq C \varepsilon.
\] 

Applying Lemma \ref{lem:solution approximation lem} to $\wt{v}$, there exists a function $h:=h_k\in C^\infty(Q_{3/4}^+)\cap C(\ol{Q}_{3/4}^+)$ satisfying that
\[
\left\{
\begin{array}{ll}
h_t-x_n^\al (a^{nn}(0,1)\Delta h)= 0
&\text{in}~Q_{3/4}^+,\\
h=\wt{v}& \text{on}~\pa_{p} Q_{3/4}^+,
\end{array}
\right.
\]
such that
\[
\|\wt{v}-h\|_{L^\infty(Q_{1/2}^+)}\leq C\varepsilon^\theta.
\]
By Proposition \ref{lem:approximation lem}, there exists a polynomial
\[
\wt{P}(x,t)=\wt{a}\cdot x_n
\]
such that
\[
\left\|h-\wt{P}\right\|_{L^\infty(Q_r^+)}\leq C r^{3-\al}\left\|h\right\|_{L^\infty(Q_1^+)}.
\]
Therefore,
\begin{align*}
\left\|\wt{v}-\wt{P}\right\|_{L^\infty(Q_r^+)}&\leq \left\|\wt{v}-h\right\|_{L^\infty(Q_r^+)}+\left\|h-\wt{P}\right\|_{L^\infty(Q_r^+)}\\
& \leq C \varepsilon^\theta+ C r^{3-\al}\\
&\leq  r^{2-\al+\beta+(1-\frac{\al}{2})\gamma},
\end{align*}
for $r$ sufficiently small and then $\varepsilon$ accordingly, since $2-\al+\beta+(1-\al/2)\gamma \leq 3-\al$.
After scaling back, we have
\[
\left|v(x,t)-P_l(x,t)-r^{l(2-\al+\beta+(1-\frac{\al}{2})\gamma)}\wt{P}\left(\frac{x}{r^l},\frac{t-1}{r^{l(2-\al)}}+1\right)\right|\leq  r^{(l+1)(2-\al+\beta+(1-\frac{\al}{2})\gamma)}
\]
for any $(x,t)\in Q_{r^{l+1}}^+$. This implies
\[
P_{l+1}(x,t)=P_l(x,t)+r^{l(2-\al+\beta+(1-\frac{\al}{2})\gamma)}\wt{P}(\frac{x}{r^l},\frac{t-1}{r^{l(2-\al)}}+1)
\]
and
\[
r^{l}|a_{l+1}-a_{l}|\leq C r^{l(2-\al+\beta+(1-\frac{\al}{2})\gamma)}.
\]
With this inequality, it follows that $a_k$ converge to $a_{\infty}$ and then $P_k$ converges to
\[
P(x,t)=a_\infty\cdot x_n-\frac{f(0,1)}{a^{nn}(0,1)(2-\al+\beta)(1-\al+\beta)}x_n^{2-\al+\beta},
\]
which satisfies
\begin{align*}
|P_k(x,t)-P(x,t)|&\leq C|x||a_k-a_\infty|\\
&\leq C |x|r^{k(2-\al+\beta+(1-\frac{\al}{2})\gamma)-k}\\
& \leq C r^{k(2-\al+\beta+(1-\frac{\al}{2})\gamma)}
\end{align*}
for all $(x,t)\in Q_{r^k}^+$. Therefore, for all $(x,t)\in Q_{r^k}^+$, we have
\[
|v(x,t)-P(x,t)|\leq |v(x,t)-P_k(x,t)|+|P_k(x,t)-P(x,t)|\leq C r^{k(2-\al+\beta+(1-\frac{\al}{2})\gamma)},
\]
which implies that
\[
|v(x,t)-P(x,t)| \leq C\left(|x|+|t-1|^{\frac{1}{2-\al}}\right)^{2-\al+\beta+(1-\frac{\al}{2})\gamma}
\]
for $(x,t)\in Q_{1/2}^+$. 
Since the choice of $z_0$ is arbitrary and the estimates are independent of the choice of $z_0$, it follows that there exist polynomials $P_{z_0}$ such that
\[
|v(x,t)-P_{z_0}(x,t)|\leq C \bigl(s[z,z_0]^\frac{2}{2-\al}\bigr)^{2-\al+\beta+(1-\frac{\al}{2})\gamma}
\]
for all $z=(x,t)\in Q_{1/4}^+(z_0)$.

\end{proof}

\begin{thm}\label{thm:generalized aij regularity }
Let $0<\gamma<\gamma^*$,  and let $v$ be a solution to \eqref{eq:equation_v0 variable}. We further assume
\[
\left\|a^{ij}\right\|_{\C^\gamma(Q_{1}^+)}+\|b^i\|_{\C^\gamma(Q_{1}^+)}+\|c\|_{\C^\gamma(Q_1^+)}\leq \frac{1}{\lambda}.
\]  
Then we have $v \in \C^{2+\gamma}(\overline{Q}_{1/2}^+)$ with
\[
\|v\|_{\C^{2+\gamma}(Q_{1/2}^+)} \leq C \left( \|v\|_{L^\infty(Q_1^+)} + \|f\|_{\C^\gamma(Q_1^+)} \right),
\]
where $C>0$ is a constant depending only on $n, \alpha, \beta, \gamma$, and $\lambda$. 
Moreover,  the quantity $x_n^{-\beta}v_t$ and $x_n^{\alpha-\beta}D^2 v$ are H\"older continuous up to the flat boundary, and their boundary traces satisfy
\[
x_n^{-\beta}v_t=0,\quad   x_n^{\alpha-\beta}D^2 v+ \operatorname{diag}\{0,\cdots,\frac{f}{a^{nn}}\}=0\quad \text{on } \partial_p Q_{1/2}^+ \cap \{x_n=0\}.
\] 
\end{thm}

\begin{proof}
For simplicity, we assume that $\|v\|_{L^\infty(Q_1^+)} + \|f\|_{\C^\gamma(Q_1^+)} \leq 1$ and only present the estimates for $v_{x_n},~x_n^{-\beta}v_t$ and $x_n^{\alpha-\beta}D^2 v$, since the estimates for $v$, $D_{x'}v$ can be obtained via analogous arguments. 

By Theorem \ref{thm:polynomial approximation}, $v$ can be approximated by polynomial $P_{z_0}$ for each $z_0\in \partial_p Q_{1/2}^+\cap \{x_n=0\}$. Let $w = v - P_{z_0}$, then we have
\[
|w(z)| \leq C \left(s[z,z_0]^{\frac{2}{2-\alpha}}\right)^{2-\alpha+\beta+(1-\frac{\alpha}{2})\gamma} \quad \text{in } Q_{1/4}^+(z_0).
\]
Considering the scaling $v_{r,z_0}$ and $w_{r,z_0}$ defined in \eqref{eq:scale at x_n=0}, this yields 
\begin{align*}
\|w_{r,z_0}\|_{C^{2,\gamma}(Q_{1/2}(0,1,1))} 
&\leq C \left(\|w_{r,z_0}\|_{L^\infty(Q_{3/4}(0,1,1))}+r^\beta \|y_n^{\beta}\wt{f}_{r,z_0}\|_{\C^{\gamma}(Q_{3/4}(0,1,1))}\right)  \\
&\leq C \left(r^{\alpha-2} \|w\|_{L^\infty(Q_{2r}^+(z_0))}+r^{\beta+\sigma} \|\wt{f}\|_{\C^{\gamma}(Q_{3/4}(0,1,1))}\right) \\
&\leq C r^{\beta+(1-\frac{\alpha}{2})\gamma},
\end{align*}
where 
\be\label{eq:inhomogeneous term of w}
\begin{aligned}
\wt{f}(z)=f(z)-&\frac{f(z_0)}{a^{nn}(z_0)}a^{nn}(z)+a(z_0)b^{n}(z)x_n^{\al-\beta}\\
&-\frac{f(z_0)}{a^{nn}(z_0)(1-\al+\beta)}b^n(z)x_n-a(z_0)c(z)x_n^{1-\beta}.  
\end{aligned}
\ee
Note that 
\[
x_n^{\alpha-\beta}D^2 v=x_n^{\alpha-\beta}D^2 w-x_n^{\alpha-\beta}D^2 P_{z_0}=x_n^{\alpha-\beta}D^2 w-  \operatorname{diag}\{0,\cdots,\frac{f}{a^{nn}}(z_0)\} ,
\] 
we obtain
\[
[y_n^{\al-\beta}D_y^2  v_{r,z_0}]_{C^\gamma(Q_{1/2}(0,1,1))} =[y_n^{\al-\beta}D_y^2  w_{r,z_0}]_{C^\gamma(Q_{1/2}(0,1,1))}\leq C r^{\beta+(1-\frac{\alpha}{2})\gamma}
\]
and
\[
|x_n^{\alpha-\beta}D^2 v(z)+\operatorname{diag}\{0,\cdots,\frac{f}{a^{nn}}(z_0)\}| = |x_n^{\al-\beta}D^2 w(z)| \leq C \left(s[z,z_0]^{\frac{2}{2-\alpha}}\right)^{(1-\frac{\alpha}{2})\gamma}.
\]
This implies that the quantity $x_n^{\alpha-\beta}D^2 w(z_0) $ is well-defined, yielding
\[
x_n^{\alpha-\beta}D^2 v(z_0) = -{ \operatorname{diag}\{0,\cdots,\frac{f}{a^{nn}}(z_0)\}}\quad \text{for all } z_0 \in \partial_p  Q_{1/2}^+ \cap \{x_n=0\}.
\]

Consequently, by varying $z_0$ over the degenerate boundary, we obtain the boundary estimate
\[
[x_n^{\alpha-\beta}D^2 v]_{\mathscr{C}_{bd}^\gamma(Q_{1/2}^+)} \leq C,
\]
and the interior estimate
\[
[x_n^{\al-\beta}D_x^2 v]_{\C_{int}^{\gamma}(Q_{1/2}^+)} = 
\sup_{\substack{z_0 \in \partial_p Q_{1/2}^+ \cap \{x_n = 0\} \\ 0 < r < 1/2 }} \frac{[y_n^{\al-\beta}D_y^2  v_{r,z_0}]_{C^\gamma(Q_{1/2}(0,1,1))}}{r^{\beta+(1-\frac{\alpha}{2})\gamma}} \leq C.
\]
Apply Lemma \ref{lem:euclidean_patching} to $x_n^{\al-\beta}D^2 v$ then gives
\[
[x_n^{\al-\beta}D^2 v]_{\C^\gamma(Q_{1/2}^+)}\leq C(n, \alpha, \gamma) \left( [x_n^{\al-\beta}D^2 v]_{\C_{bd}^{\gamma}(Q_{1/2}^+)} + [x_n^{\al-\beta}D^2 v]_{\C_{int}^{\gamma}(Q_{1/2}^+)} \right)\leq C.
\]

Similarly, by noting that $x_n^{-\beta}v_t=x_n^{-\beta}w_t$, we derive that 
\[
|x_n^{-\beta}v_t(z)| = |x_n^{-\beta}w_t(z)| \leq C \left(s[z,z_0]^{\frac{2}{2-\alpha}}\right)^{(1-\frac{\alpha}{2})\gamma}
\]
and
\[
[y_n^{-\beta}\partial_tv_{r,z_0}]_{C^\gamma(Q_{1/2}(0,1,1))} =[y_n^{-\beta}\partial_tw_{r,z_0}]_{C^\gamma(Q_{1/2}(0,1,1))}\leq C r^{\beta+(1-\frac{\alpha}{2})\gamma}.
\]
By varying $z_0$ over the degenerate boundary and recalling Lemmas \ref{lem:euclidean_patching} and \ref{lem:interior_seminorm_scaling}, we obtain that
\[
\|x_n^{-\beta}v_t\|_{\C^\gamma(Q_{1/2}^+)} \leq C \quad \text{and} \quad x_n^{-\beta}v_t  = 0 \quad \text{on } \partial_p Q_{1/2}^+ \cap \{x_n=0\}.
\]

Finally, we show that  $v_{x_n}\in \C^\gamma(Q_{1/2}^+)$.
Notice that
\[
w_{r,z_0}(z)=v_{r,z_0}(z)-a(z_0) r^{\al-1} x_n+\frac{f(z_0)r^\beta}{a^{nn}(z_0)(2-\al+\beta)(1-\al+\beta)}x_n^{2-\al+\beta},
\]
then
\[
(w_{r,z_0})_{x_n}(z)=(v_{r,z_0})_{x_n}(z)-a(z_0) r^{\al-1}+\frac{f(z_0)r^\beta}{a^{nn}(z_0)(1-\al+\beta)}x_n^{1-\al+\beta}.
\]
Hence, we obtain that
\begin{align*}
[(v_{r,z_0})_{x_n}]_{C^\gamma(Q_{1/2}(0,1,1))}&=[(w_{r,z_0})_{x_n}]_{C^\gamma(Q_{1/2}(0,1,1))}
+\frac{f(z_0)r^{\beta}}{a^{nn}(z_0)(1-\al+\beta)}[y_n^{1-\al+\beta}]_{C^\gamma(Q_{1/2}(0,1,1))} \\
&\leq C r^{\beta}.
\end{align*}
Since $w=v-P_{z_0}$, we have 
\[
w_{x_n}(z)=v_{x_n}(z)-a(z_0)+\frac{f(z_0)}{a^{nn}(z_0)(1-\al+\beta)}x_n^{1-\al+\beta},
\]
so by the estimate of $w$, it follows that $v_{x_n}(z_0)=a(z_0)$, and
\[
[v_{x_n}]_{\C_{bd}^\gamma(Q_{1/2}^+)}\leq [w_{x_n}]_{\C_{bd}^\gamma(Q_{1/2}^+)}+C[x_n^{1-\al+\beta}]_{\C_{bd}^\gamma(Q_{1/2}^+)}\leq C.
\]
By varying $z_0$ over the degenerate boundary, we derive
\[
[v_{x_n}]_{\C_{int}^{\gamma}(Q_{1/2}^+)} = 
\sup_{\substack{z_0 \in \partial_p Q_{1/2}^+ \cap \{x_n = 0\} \\ 0 < r < 1/2 }} \frac{[(v_{r,z_0})_{x_n}]_{C^\gamma(Q_{1/2}(0,1,1))}}{r^{\al-1+(1-\frac{\alpha}{2})\gamma}} \leq C,
\]
and then, applying Lemma \ref{lem:euclidean_patching}, the conclusion follows.
\end{proof}

\section{Degenerate equations on general domain}\label{sec:proofofmainthm}

In the preceding sections, we established the boundary regularity for the linear degenerate equation with variable coefficients on $Q_1^+$, as given in Section \ref{sub:variable_coefficients}. The objective of the current section is to prove Theorem \ref{thm: existence and uniqueness of general linear equation} by applying previous variable-coefficient estimates.

Recall that the global initial-boundary value problem on a bounded domain $\Omega \subset \mathbb{R}^n$ with a smooth boundary is
\[
    L v = v_t - \omega^\alpha \left( a^{ij}v_{ij} + b^i v_i \right) + cv = \omega^\beta f \quad \text{in } \Omega \times (0,T],
\]
subject to the homogeneous Dirichlet condition on the boundary
\[
    v = 0 \quad \text{on } \partial \Omega \times [0,T],
\]
and  initial datum $g_0 \in C_0(\overline{\Omega})$ vanishes on the boundary $\partial\Omega$. Here, the degeneracy of the operator is captured by the weight function $\omega(x)$, which behaves comparably to the distance function near the boundary; that is,
\[
\lambda d \leq \omega \leq \lambda^{-1} d
\]
for some \(\lambda\in(0,1]\)

for all $x \in \Omega$ sufficiently close to $\partial\Omega$. For this general problem, we shall establish the global existence of solutions, as well as their boundary and global regularity properties.

Before presenting the main result of this section, let us briefly recall our basic hypotheses. Recall that the relevant parameters satisfy
\[
\al\in(1,2),\quad\beta\in (\al-1,\al/2]\quad \text{and}\quad \gamma^*=\frac{2-2\al+2\beta}{2-\al}.
\]
Throughout this section, we assume that the structural conditions \eqref{eq:structure_conditions} hold. In addition, for every $T>0$, all coefficients $a^{ij}$, $b^i$, $c$, together with the source term $f$, are locally H\"older continuous in $\Omega\times(0,T]$.

\begin{thm}\label{thm:existence in thm 1} 
For any initial datum $g_0 \in C_0(\overline{\Omega})$ and any given $T > 0$, the initial-boundary value problem \eqref{eq:general linearized equation} admits a unique classical solution.
\end{thm}

\begin{proof}
Consider the approximation equation
\[
\left\{
\begin{array}{ll}
L_\va v_\va:=(v_\va)_t-(\omega+\va)^\al (a^{ij}D_{ij} v_\va+b^i D_i v_{\va})+cv_\va=(\omega+\va)^\beta f
&\text{in}~\Omega\times(0,T],\\
v_\va=0 & \text{on}~\partial\Omega\times(0,T],\\
v_\va=g_0& \text{on}~\Omega\times\{0\},
\end{array}
\right.
\]
and by the classical theory for uniformly parabolic equations, for each $\va>0$ there exists a unique solution
\[
v_\va\in C(\ol{\Omega}\times [0,T]) \cap C^{2,1}_{loc}(\Omega\times(0,T]).
\]
Moreover, by the interior Schauder estimates, for every compact set $K\subset\subset \Omega\times (0,T]$,
\[
\|v_\va\|_{C^{2,\gamma}(K)}\leq C_K.
\]
We next show that, for suitable positive constants $\mu$, $M_\tau$, and $A$, one has, for all $\va<1$,
\begin{equation}\label{eq:c0estimate for vva}
|v_\va(x,t)|\le e^{\mu t}\inf_{\tau>0} \Bigl( M_\tau A ((\omega+\va)^{2-\al}-\va^{2-\al})+\kappa_0(\tau)\Bigr)
\quad \text{in } \om\times [0,T],
\end{equation}
where 
\[
\kappa_0(s) := \|g_0\|_{L^\infty(\om \cap \{0 < d(x) < s\})}\quad  \text{and}\quad  M_{\tau}= \sup_{ \delta> \tau }  \frac{\kappa_0(\delta) }{\delta^{2-\alpha}} +1,
\]
which implies that $v_\va$ is uniformly continuous up to the parabolic boundary.  Then, by taking $\tau=\omega^{1/2}$ and noting that $\sup_{ \delta> \tau }  \frac{\kappa_0(\delta) }{\delta^{2-\alpha}}\leq \omega^{\frac{\al}{2}-1}$, we have
\[
|v_\va(x,t)|\leq C \Bigl( \omega^{\frac{\al}{2}}(x)+\kappa_0(\omega^{1/2}(x))\Bigr).
\]
The right hand side decays to $0$ as $x\to \pa\om$ since $\omega\to 0$.  Since the operator is uniformly parabolic away from the boundary, the standard theory ensures $\{v_\va\}$ a uniform modulus of continuity. It then follows from the Arzel\`a--Ascoli theorem that there exist a sequence $\va_k\to 0$ and a continuous function $v$ such that $v_{\va_k}\to v$ uniformly. By the maximum principle, such a limit is unique, then the limit function $v$ solves the initial-boundary value problem \eqref{eq:general linearized equation}. Therefore, the existence and uniqueness follows.

Next, we show \eqref{eq:c0estimate for vva}. Let 
\[
\varphi(t)=\frac{M}{\mu}(e^{\mu t}-1)+\|g_0\|_{L^\infty(\om)},
\]
where $M=(\operatorname{diam}(\om)+1)^\beta\left\| f\right\|_{L^\infty(\om\times [0,T])}$ and $\mu>0$ is a constant satisfying $\mu+c(x,t) > 0$ for all $(x,t)\in \om\times(0,T]$.
Then 
\begin{align*}
&\varphi_t - (\omega+\va)^\beta (a^{ij}\varphi_{ij}+b^i\varphi_i)+c\varphi\\
&=\frac{M}{\mu}(\mu e^{\mu t}+c(e^{\mu t}-1))+\|g_0\|_{L^\infty(\om)}\geq M \geq (\omega+\va)^\beta f
\end{align*}
and $\varphi\geq 0$ on $\pa \om\times (0,T]$, $\varphi(0)\geq g_0$ on $\om\times \{t=0\}$. By the comparison principle, we have 
\[
v_\va \leq \frac{M}{\mu}(e^{\mu t}-1)+\|g_0\|_{L^\infty(\om)}\leq \frac{M}{\mu}e^{\mu T}+\|g_0\|_{L^\infty(\om)}=:C_T.
\]
This gives an upper bound of $v_\va$. Define
\[
\psi(x,t)=A\bigl((\omega+\va)^{2-\al}-\va^{2-\al}\bigr),
\]
then
\begin{align*}
\psi_{x_i}&=A (2-\al)(\omega+\va)^{1-\al}\omega_{x_i},\\
\psi_{x_ix_j}&=A(2-\al)(1-\al)(\omega+\va)^{-\al}\omega_{x_i}\omega_{x_j}+A(2-\al)(\omega+\va)^{1-\al}\omega_{x_ix_j},
\end{align*}
and
\begin{align*}
&\psi_t-(\omega+\va)^\al \left(a^{ij}\psi_{ij}+b^i \psi_i\right)\\
=&A(2-\al)(\al-1)a^{ij}\omega_{x_i}\omega_{x_j}-A(2-\al)(\omega+\va) a^{ij}\omega_{x_ix_j}- A (2-\al)(\omega+\va) b^i\omega_{x_i}\\
\geq & A(2-\al)(\al-1) \lambda |D\omega|^2-A(2-\al)(\omega+\va) a^{ij}\omega_{x_ix_j}- A (2-\al)(\omega+\va) b^i\omega_{x_i}.
\end{align*}
Since $\inf_{\om} (\omega/d)>0$, we have $|D\omega|\geq C_1$ for some positive constant $C_1$ in $\{\omega<\delta_1\}$ with $\delta_1>0$, then 
\[
a^{ij}\omega_{x_i}\omega_{x_j}\geq \lambda |D\omega|^2 \geq \lambda C_1^2.
\]
Therefore, for a fixed sufficiently small $\delta_1>0$ and choosing $A$ sufficiently large, we obtain
\[
L_\va \psi-c\psi \ge (\omega+\va)^\beta f
\qquad \text{in } \{\omega<\delta_1\}\times [0,T],
\]
and
\[
v_\va \le C_T \le \psi
\qquad \text{on } \{\omega\ge \delta_1\}\times [0,T].
\]
As defined in \eqref{eq:def_omega_0}, let 
\[
\Psi_{\tau}(x,t)=e^{\mu t} (M_\tau \psi(x)+\kappa_0(\tau)).
\]
Then 
\[
L_\va \Psi_\tau=e^{\mu t} M_{\tau} (L_{\va}\psi-c\psi)+(\mu+c)\Psi_\tau\geq (\omega+\va)^\beta f.
\]
Notice that $\Psi_\tau\geq 0=v_\va$ on $\pa\om\times (0,T]$ and for sufficiently large $A$, we have 
\[
\Psi_\tau(x,0) \geq \frac{A}{2}\tau^{2-\al} \geq \| g_0\|_{L^\infty(\om)} \quad \text{in}~\om\cap \{d(x)\geq \tau\} 
\]
and 
\[
\Psi_\tau(x,0) \geq \kappa_0(\tau)\geq g_0(x)\quad \text{in}~\om\cap \{0<d(x)<\tau\}.
\]
The comparison principle then yields
\[
v_\varepsilon(x,t) \leq \Psi_{\tau}(x,t) \quad \text{in } \om\times [0,T].
\] 
By symmetry, $v_\varepsilon(x,t) \geq -\Psi_{\tau}(x,t)$ holds as well. Then we obtain that
\[
|v_\va(x,t)|\leq e^{\mu t}\inf_{\tau>0} \Bigl( M_\tau A ((\omega+\va)^{2-\al}-\va^{2-\al})+\kappa_0(\tau)\Bigr).
\]
\end{proof}

\begin{thm}\label{thm:localize the schauder estimates in time variable}
Let $\om,\al,\beta,\gamma,a^{ij},b^i,c$ and $f$ be given as in Theorem \ref{thm: existence and uniqueness of general linear equation}. Let $v $ satisfy \eqref{eq:general linearized equation}. Then $v\in\mathscr{C}^{2+\gamma}(\overline{\Omega}\times (0,T])$ and  there exists $C>0$ which depends only on $n,\al,\beta,\lambda,\gamma,\om$ and the $\mathscr{C}^{\gamma}(\overline{\Omega}\times (0,T])$ norms of $a^{ij},b^i,c$ such that
\[
\left\|v\right\|_{\mathscr{C}^{2+\gamma}(\om\times [\frac{T}{2},T])}\leq C\left(\left\|v\right\|_{L^\infty(\om\times [0,T])}+\left\|f\right\|_{\mathscr{C}^\gamma(\om\times [0,T])}\right).
\]
\end{thm}

\begin{proof}
We first decompose the compact domain $\om$ into a finite union of compact subdomains,
\[
\om=\om^0 \cup \left(\bigcup_{l\geq 1} \om^l \right),
\]
in such a way that
\[
\operatorname{dist}(\om^0,\pa\om)\geq \frac{\rho}{2}>0,
\]
and, for every $l\geq 1$,
\[
\om^l=B_{\rho}(x_l)\cap \ol{\om},
\]
where $B_{\rho}(x_l)$ denotes the ball of radius $\rho>0$ centered at some point $x_l\in \pa\om$.
On the interior subdomain $\om^0$, the operator $L$ is uniformly parabolic. Hence, the classical Schauder theory for linear parabolic equations applies on $\om^0$, since the degenerate distance is equivalent there to the standard Euclidean distance.

Next, we focus on the domains $\om^l$, $l\geq 1$, which are close to the boundary of $\om$. For $\mu>0$, denote by $\ol{Q}_{\mu}$ the cylinder
\[
\ol{Q}_{\mu}=\ol{B}_{\mu}^+\times [0,\mu^{2-\al}],
\]
where $\ol{B}_{\mu}^+=\left\{(x_1,\dots,x_n)\in B_{\mu}(0)~\middle|~x_n\geq 0\right\}$ is a half ball. We select smooth charts $\Upsilon_l: \ol{B}_{\mu}^+\rightarrow \ol{\om}^l $ such that they map $\ol{B}_{\mu}^+\cap \{x_n=0\}$  onto $\ol{\om}^l\cap \pa\om$ and $\Upsilon_l(0)=x_l$. This is possible if $\rho$ is chosen sufficiently small. Under the change of coordinates induced by the charts $\Upsilon_l$, the operator, restricted on each $\ol{\om}^l\times[0,\mu^{2-\al}]$, is transformed to an operator $\ol{L}_l$ of the form
\[
\ol{L}_l \ol{v}=\ol{v}_t-x_n^\al (\ol{a}^{ij}_l \ol{v}_{ij}+\ol{b}_l^i\ol{v}_i)+\ol{c}_l\ol{v}=x_n^\beta \ol{f}
\]
defined on $\ol{Q}_{\mu}$.

Moreover, by composing the flattening map with a suitable transformation, the charts $\Upsilon_l$ can be chosen appropriately so that the coefficients of $\ol{L}_l$ satisfy
\[
\ol{a}_l^{ij}\xi_i\xi_j\geq \lambda |\xi|^2>0\quad \forall ~\xi\in \R^n\setminus \{0\}
\]
and
\[
\|\ol{a}_l^{ij}\|_{\C^\gamma(\ol{Q}_{\mu})}+\|\ol{b}_l^i\|_{\C^\gamma(\ol Q_{\mu})}+\|\ol{c_l}\|_{\C^\gamma(\ol Q_{\mu})}\leq \frac{1}{\ol{\lambda}},
\]
for some positive constant $\ol{\lambda}$ and $\lambda$.

By scaling on $\ol{v}$ similar to \eqref{eq:scale at x_n=0}, we have terms $a_\mu^{ij}$, $b_{\mu}^i$, $c_\mu$ and $f_\mu$ similar to \eqref{eq:scaled coefficients}.  Therefore, we obtain that
\[
\|v_\mu\|_{\C^{2+\gamma}(Q_{1/2}^+)}\leq C \left( \left\|v_\mu \right\|_{L^\infty(Q_{1}^+)}+\|f_{\mu}\|_{\C^\gamma(Q_{1}^+)}\right).
\]
Scaling back $v_{\mu}$, we have 
\[
\|\ol{v}\|_{\C^{2+\gamma}(B_{\mu/2}^+\times (\mu^{2-\al}-(\mu/2)^{2-\al},\mu^{2-\al}])}\leq C \left( \left\|\ol{v}\right\|_{L^\infty(\ol{Q}_{\mu})}+\|\ol{f}\|_{\C^\gamma(\ol{Q}_{\mu})}\right),
\]
and
\[
x_n^{-\beta} \ol{v}_{t}(z_0)=0 \quad \text{and}\quad  x_n^{\alpha-\beta}D^2 \ol{v}(z_0) = -{ \operatorname{diag}\{0,\cdots,\frac{\ol{f}}{\ol{a}^{nn}}(z_0)\}}
\]
for all  $z_0 \in \partial_pQ_{1/2}^+ \cap \{x_n=0\}$. Since $\pa\Omega$ is smooth, then $\Upsilon_l$ is smooth, so by finite covering,  we obtain that
\[
\|v\|_{\C^{2+\gamma}(\Omega^l\times [\frac{T}{2},T])}\leq C \left( \left\|v\right\|_{L^\infty(\Omega^l\times [0,T])}+\|f\|_{\C^{\gamma}(\Omega^l\times [0,T])}\right).
\]
Combining the estimates on each $\om^l$, we obtain that
\[
\|v\|_{\C^{2+\gamma}(\Omega\times [\frac{T}{2},T])}\leq C \left( \left\|v\right\|_{L^\infty(\Omega\times [0,T])}+\|f\|_{\C^{\gamma}(\Omega\times [0,T])}\right).
\]
Moreover, through the inverse flattening map, and noting that 
\[
\overline{f}(\cdot,t)=(\frac{\omega}{d})^\beta f(\Upsilon_l \cdot,t)\quad \ol{a}^{ij}(\cdot,t)=(\frac{\omega}{d})^\al a^{ij}(\Upsilon_l \cdot,t)
\] 
we have 
\[
\omega^{-\beta} v_{t}=0 \quad \text{and}\quad  \omega^{\alpha-\beta}D^2 v = -\frac{f}{a^{\nu\nu}}\nu \otimes\nu \quad \text{on } \partial\Omega\times(0,T].
\]
\end{proof}

\begin{lem}\label{lem:a supersolution up to the bottom}
Let $Q_r^0=B_r^+\times (0,r^{2-\alpha})$, $\sigma=(1-\al/2)\gamma$ and let $v,f$ be as in Theorem \ref{thm:generalized aij regularity }. If $f(x',0,0)=0$ for all $|x'|<1$ and $v(x,0)\equiv 0$, then
\[
|v|\leq w(x,t)=2Ntx_n^{\beta+\sigma}+M|x|^2\quad \text{in}~Q_{1/2}^0
\]
for some large constants $M$ and $N$. Moreover, $v\in \C^{2+\gamma}(Q_{1/2}^0)$ with 
\[
\|v\|_{\C^{2+\gamma}(Q_{1/2}^0)}\leq C \|f\|_{\C^\gamma(B_1^+\times (0, 2^{\al-2}))}.
\]
\end{lem}

\begin{proof}
By the equation, we compute that 
\begin{align*}
&w_t-x_n^\al \left(a^{ij}D_{ij}w+b^iD_i w\right)+cw\\
=& 2Nx_n^{\beta+\sigma}-x_n^\alpha \Bigl(2 a^{nn}Nt (\beta+\sigma) (\beta+\sigma-1)x_n^{\beta-2+\sigma}+2M\sum_{i=1}^{n}a^{ii}\Bigr)\\
& - x_n^\alpha \Bigl(2b^n Nt (\beta+\sigma)x_n^{\beta+\sigma-1}+2Mb^i x_i\Bigr)+c \Bigl(2Nt x_n^{\beta+\sigma}+M|x|^2\Bigr)\\
=& \Bigl(2N+2cNt\Bigr)x_n^{\beta+\sigma}+2a^{nn}Nt(\beta+\sigma)(1-\beta-\sigma)x_n^{\al+\beta+\sigma-2}+cM|x|^2\\
&-2b^n Nt (\beta+\sigma)x_n^{\al+\beta+\sigma-1}-\Bigl(2M \sum_{i=1}^{n}a^{ii}+2M b^i x_i \Bigr)x_n^\al.
\end{align*}
By $\C^\gamma$-regularity of $f$ and $f(x',0,0)=0$ for all $|x'|<1$, we have 
\[
v_t-x_n^\al \left(a^{ij}D_{ij}v+b^iD_i v\right)+cv=x_n^\beta f=x_n^\beta (f(x,t)-f(x',0,0))\leq Cx_n^\beta (x_n^{\sigma}+t^{\frac{\sigma}{2-\al}}).
\]
Notice that there exists $r>0$ such that $|b^n|(\beta+\sigma) r+|c| r^{2-\al}\leq 1/2$, then by choosing $N$ sufficiently large, depending on $M$, we have 
\[
w_t-x_n^\al \left(a^{ij}D_{ij}w+b^iD_i w\right)+cw\geq C (x_n^{\beta+\sigma}+x_n^\beta t^{\frac{\sigma}{2-\al}})  \quad \text{in}~Q_r^0.
\]
Indeed, it suffices to show that 
\[
C x_n^{\beta}t^{\frac{\sigma}{2-\al}}\leq N x_n^{\beta+\sigma}+2a^{nn}Nt(\beta+\sigma)(1-\beta-\sigma)x_n^{\al+\beta+\sigma-2}.
\]
When $0<t<x_n^{2-\al},$ it is trivial. When $t>x_n^{2-\al}$,  we have 
\[
t^{\frac{\sigma}{2-\al}}x_n^\beta =tx_n^\beta\cdot t^{\frac{\al+\sigma-2}{2-\al}}\leq tx_n^\beta\cdot x_n^{\al+\sigma-2}
\]
since $\sigma<2-\alpha$. When $t=0$ or $x_n=0$,  $v=0\leq w$. On $\pa_p Q_r^0 \cap \{x_n>0\}$, by choosing $M$ such that $Mr^2\geq  \left\|v\right\|_{L^\infty(Q_r^0)}$, we have $v\leq w$.  Therefore, by the comparison principle, we have $v\leq w$ in $Q_r^0$. Applying the same argument to $-v$  gives  $|v|\le w$ in $Q_r^0$. 

Then by scaling
\[
|v(x,t)|\leq C \bigl( |x|+|t|^{\frac{1}{2-\alpha}}\bigr)^{2-\alpha+\beta+\sigma}\quad \text{for}~(x,t)\in Q_{1/2}^0,
\]
from which, we have for each fixed $z_0=(x',0,0) \in \partial_p Q_{1/2}^0 \cap \{x_n=0\}$,
\[
|v(z)| \leq C \left(s[z,z_0]^{\frac{2}{2-\alpha}}\right)^{2-\alpha+\beta+\sigma} \quad \text{in } Q_{1}^0.
\] 
Furthermore, combining this with the interior estimate in
Theorem \ref{thm:polynomial approximation}, we obtain that, for each
fixed $z_0\in \partial_p Q_{1/2}^0\cap\{x_n=0\}$, with $P_{z_0}$ defined as in
Theorem \ref{thm:polynomial approximation},
\[
|v(z)-P_{z_0}(z)|
\leq C\left(s[z,z_0]^{\frac{2}{2-\alpha}}\right)^{
2-\alpha+\beta+\sigma}
\quad \text{in } Q_{1/2}^+(z_0).
\]
Here and below, we suppress the intersection with the region where the relevant solution is well-defined. Thus, each domain $A$ should be understood as $A\cap Q_1^0$; after rescaling, $Q_1^0$ is replaced by its corresponding rescaled image.

Let $w(z)=v(z)-P_{z_0}(z)$ and consider the scaling $w_{r,z_0}, v_{r,z_0}$ defined in \eqref{eq:scale at x_n=0}. Since $g\equiv 0$, \eqref{eq:classical schauder} yields
\begin{align*}
\|w_{r,z_0}\|_{C^{2,\gamma}(Q_{1/2}(0,1,1))} 
&\leq C \left(\|w_{r,z_0}\|_{L^\infty(Q_{3/4}(0,1,1))}+r^\beta \|y_n^{\beta}\wt{f}_{r,z_0}\|_{\C^{\gamma}(Q_{3/4}(0,1,1))}\right)  \\
&\leq C \left(r^{\alpha-2} \|w\|_{L^\infty(Q_{2r}^+(z_0))}+r^{\beta+\sigma} \|\wt{f}\|_{\C^{\gamma}(Q_{3/4}(0,1,1))}\right) \\
&\leq C r^{\beta+\sigma},
\end{align*}
where $\wt{f}$ is defined in \eqref{eq:inhomogeneous term of w}. Then we obtain that
\[
[y_n^{\al-\beta}D_y^2  v_{r,z_0}]_{C^\gamma(Q_{1/2}(0,1,1))}=[y_n^{\al-\beta}D_y^2  w_{r,z_0}]_{C^\gamma(Q_{1/2}(0,1,1))} \leq C r^{\beta+\sigma}
\]
and
\[
|x_n^{\al-\beta}D^2 v(z)+ \operatorname{diag}\{0,\cdots,\frac{f}{a^{nn}}(z_0)\}|=|x_n^{\alpha-\beta}D^2 w(z)| \leq C \left(s[z,z_0]^{\frac{2}{2-\alpha}}\right)^{\sigma}.
\]
This implies that the quantity $x_n^{\alpha-\beta}D^2 w(z_0) $ is well-defined, yielding
\[
x_n^{\alpha-\beta}D^2 w(z_0) = 0\quad \text{for all } z_0 \in\partial_p  Q_{1/2}^0 \cap \{x_n=0\}.
\]
Consequently, by varying $z_0$ over the degenerate boundary, we obtain the boundary estimate
\[
[x_n^{\alpha-\beta}D^2 v]_{\mathscr{C}_{bd}^\gamma(Q_{1/2}^+(z_0))} \leq C,
\]
and the interior estimate
\[
[x_n^{\al-\beta}D_x^2 v]_{\C_{int}^{\gamma}(Q_{1/2}^+(z_0))} = 
\sup_{\substack{z_0 \in \partial_p Q_{1/2}^0 \cap \{x_n = 0\} \\ 0 < r < 1/2 }} \frac{[y_n^{\al-\beta}D_y^2  v_{r,z_0}]_{C^\gamma(Q_{1/2}(0,1,1))}}{r^{\beta+\sigma}} \leq C
\]
Apply Lemma \ref{lem:euclidean_patching} to $x_n^{\al-\beta}D^2 v$ then gives
\begin{align*}
&[x_n^{\al-\beta}D^2 v]_{\C^\gamma(Q_{1/2}^+(z_0))}\\
\leq  &C(n, \alpha, \gamma) \left( [x_n^{\al-\beta}D^2 v]_{\C_{bd}^{\gamma}(Q_{1/2}^+(z_0))} + [x_n^{\al-\beta}D^2 v]_{\C_{int}^{\gamma}(Q_{1/2}^+(z_0))} \right)\\
\leq & C. 
\end{align*}

The H\"{o}lder estimates of $v,Dv,x_n^{-\beta}v_t$ can be obtained by similar ways.
\end{proof}

\begin{proof}[\textbf{Proof of Theorem \ref{thm: existence and uniqueness of general linear equation}}]
We omit $\al,\beta,\Omega$ in the subscript of space $\C_{\al,\beta,\Omega}^{2+\gamma}$. The existence and uniqueness follows by Theorem \ref{thm:existence in thm 1} and  the classical maximum principle. And the result  of $(i)$ follows from Theorem \ref{thm:localize the schauder estimates in time variable}.

Next, we prove $(ii)$. Let $\ol{v}(x,t):=v(x,t)-g_0(x)$ and 
\[
\ol{f}:=f+\omega^{\alpha-\beta}(a^{ij}g_{0,ij} + b^ig_{0,i}) - c \omega^{-\beta}g_0.
\]
It is elementary to check that
\[
\omega^{1-\beta}\in \mathscr{C^\gamma_{\al,\beta}}(\om\times [0,T])\quad\mbox{if }\gamma\le \frac{2-2\beta}{2-\alpha}.
\]
Therefore, this $\ol{f}$ belongs to $\mathscr{C^\gamma_{\al,\beta}}(\om\times [0,T])$ and satisfies $\ol{f}=0$ on $\pa\om\times \{t=0\}$ by \eqref{eq:compatibility condition t=0} and $g\in \C_{\al,\beta}^{2+\gamma}(\om)$. Combining the argument in the proof of Theorem \ref{thm:localize the schauder estimates in time variable}, Lemma \ref{lem:a supersolution up to the bottom} and a rescaling argument lead to
\[
\|\bar v\|_{\C^{2+\gamma}(\om\times [0,T_1])}\leq C \left( \left\|\bar v\right\|_{L^\infty(\om\times [0,T])}+\|\bar f\|_{\C^\gamma(\om\times [0,T])}\right)
\]
for some $T_1\leq T$. Combining the above estimates with Theorem \ref{thm:localize the schauder estimates in time variable}, we obtain part (ii) of Theorem \ref{thm: existence and uniqueness of general linear equation}.
\end{proof}

\small


\bigskip

\smallskip

\noindent T. Jin

\noindent Department of Mathematics, The Hong Kong University of Science and Technology\\
Clear Water Bay, Kowloon, Hong Kong\\
Email: \textsf{tianlingjin@ust.hk}

\medskip

\noindent X. Tu

\noindent Department of Mathematics, The Hong Kong University of Science and Technology\\
Clear Water Bay, Kowloon, Hong Kong\\[1mm]
Email:  \textsf{maxstu@ust.hk}, \textsf{tuxushan28561m@gmail.com}

\medskip

\noindent J. Xiong

\noindent School of Mathematical Sciences, Laboratory of Mathematics and Complex Systems, MOE\\ Beijing Normal University, 
Beijing 100875, China\\
Email: \textsf{jx@bnu.edu.cn}

\medskip

\noindent Z. Zheng

\noindent Department of Mathematics, The Hong Kong University of Science and Technology\\
Clear Water Bay, Kowloon, Hong Kong\\[1mm]
Email:  \textsf{zzhengax@connect.ust.hk}

\end{document}